\documentclass[12pt]{amsart}
\usepackage[all]{xy}
\usepackage{mathrsfs}
\usepackage[margin=0.75in]{geometry}

\usepackage{amsmath,amsfonts,amssymb,latexsym,stmaryrd}
\usepackage{graphicx}
\usepackage[cyr]{aeguill}\usepackage[francais,english]{babel}

\usepackage{color}
\usepackage[utf8]{inputenc}
\usepackage{comment}
\usepackage[shortlabels]{enumitem}
\usepackage{etoolbox}
\usepackage{float}
\usepackage{latexsym}
\usepackage{lipsum}
\usepackage{needspace}
\usepackage{tikz}
\usepackage{hyperref}

\newtheorem{theorem}{Theorem}
\newtheorem*{Mtheorem}{Main Theorem}

\numberwithin{theorem}{section}
\numberwithin{equation}{section}
\newtheorem{lemma}[theorem]{Lemma}

\newtheorem{proposition}[theorem]{Proposition}
\newtheorem{definition}[theorem]{Definition}
\newtheorem{corollary}[theorem]{Corollary}
\newtheorem{claim}[theorem]{Claim}

\newtheorem{question}[theorem]{Question}

\theoremstyle{remark}
\newtheorem{remark}[theorem]{Remark}

\newtheorem*{notation}{Notation}
\newtheorem*{thank}{Acknowledgments}

\newenvironment{myitemize}{
  \begin{itemize}
    \setlength{\itemsep}{3 pt}
    \setlength{\parsep}{0 pt}
}{
  \end{itemize}
}

\renewcommand{\mod}{\bmod\,}
\newcommand{\di}{\textnormal{diam}\,}
\newcommand{\gi}{\textnormal{girth}\,}

\newcommand{\N}{\mathbb{N}}
\newcommand{\Z}{\mathbb{Z}}
\newcommand{\legendre}[2]{\genfrac{(}{)}{}{}{#1}{#2}}

\begin{document}
	
\title{Logarithmic girth expander graphs of $SL_n(\mathbb F_p)$}
\author{Goulnara Arzhantseva}
\address{Universit\"at Wien, Fakult\"at f\"ur Mathematik\\
Oskar-Morgenstern-Platz 1, 1090 Wien, Austria.}
\email{\href{mailto:goulnara.arzhantseva@univie.ac.at}{goulnara.arzhantseva@univie.ac.at}}

\author[Arindam Biswas]{Arindam Biswas}
\address{Universit\"at Wien, Fakult\"at f\"ur Mathematik, Oskar-Morgenstern-Platz 1, 1090 Wien, Austria}
%\curraddr{}
\email{arin.math@gmail.com}

\date{}
\thanks{This research was partially supported by the European Research Council (ERC) grant of Goulnara Arzhantseva, ``ANALYTIC'' grant agreement no.\ 259527.}

\begin{abstract}We provide an explicit construction  of finite 4-regular graphs $(\Gamma_k)_{k\in \mathbb N}$ with ${\gi \Gamma_k\to\infty}$ as $k\to\infty$ and 
$\frac{\di \Gamma_k}{\gi \Gamma_k}\leqslant D$ for some $D>0$ and all $k\in\N$.
For each fixed dimension $n\geqslant 2,$ we find a pair of matrices in $SL_{n}(\mathbb{Z})$ such that (i) they generate a free subgroup, (ii)~their reductions $\bmod\, p$  generate $SL_{n}(\mathbb{F}_{p})$ for all sufficiently large primes $p$, 
 (iii) the corresponding Cayley graphs of $SL_{n}(\mathbb{F}_{p})$ have girth at least $c_n\log p$ for  some $c_n>0$.
Relying on growth results (with no use of expansion properties of the involved graphs), we observe that the diameter of those Cayley graphs is at most $O(\log p)$. This gives 
infinite sequences of finite $4$-regular Cayley graphs of $SL_n(\mathbb F_p)$ as $p\to\infty$ with large girth and bounded diameter-by-girth ratio. These are the first explicit examples in all dimensions $n\geqslant 2$ (all prior examples were in $n=2$). Moreover, they happen to be expanders. Together with  Margulis' and Lubotzky-Phillips-Sarnak's classical constructions, these new graphs are the only known explicit logarithmic girth Cayley graph expanders.
\end{abstract}
\maketitle

%%%%%%%%%%%%%%%%%%%%%%%%%%%%%%	
\section{Introduction}
The \emph{girth} of a graph is the edge-length of its shortest non-trivial cycle (it is assigned to be infinity for an acyclic graph). The \emph{diameter} of a graph 
is the greatest edge-length distance between any pair of its vertices. We regard a countable graph $\Gamma$ as a sequence of its connected components 
$\Gamma=(\Gamma_k)_{k\in\mathbb{N}}$ each of which is endowed with the edge-length distance.

\begin{definition}[large girth graph, cf.~\cite{Biggs98}]
A graph $\Gamma=(\Gamma_k)_{k\in\mathbb{N}}$  is \emph{large girth} if $\gi \Gamma_k\to \infty$, and it is \emph{logarithmic girth} if 
there exists a constant $c>0$ such that for all  $k\in\mathbb N$:
$${\gi \Gamma_k}\geqslant c\log\vert \Gamma_k\vert.$$ 
\end{definition}

\begin{definition}[{\rm dg}-bounded graph~\cite{AT18}]
A graph $\Gamma=(\Gamma_k)_{k\in\mathbb{N}}$ is \emph{dg-bounded} if there exists a constant $D>0$ such that  for all $k\in\mathbb N$:
$$\frac{\di \Gamma_k}{\gi \Gamma_k}\leqslant D.$$ 
\end{definition}
  A dg-bounded graph with uniformly bounded degree $r$ satisfies $\gi \Gamma_k \geqslant \frac{1}{D}\di \Gamma_k $ and $\di \Gamma_k\geqslant \log_{r}(\vert \Gamma_k\vert-1),$ for $r\geqslant 3$, while  $\di \Gamma_k\geqslant  \frac{1}{2}(\vert \Gamma_k\vert-1)$ for $r=2$.
Hence, for such a graph, $\gi \Gamma_k\to\infty,$ whenever  $\vert \Gamma_k\vert\to \infty$ as  $k\to\infty.$ \medskip

In this paper, we focus on  large girth dg-bounded graphs with uniformly bounded degree $r\geqslant 2$. 
For $r=2$, a growing sequence of cycle graphs is an easy example of  such a graph.
A major theoretical and practical challenge is to build large girth dg-bounded graphs for $r\geqslant 3$. \medskip

The \emph{existence} of $r$-regular large girth dg-bounded graphs $\Gamma=(\Gamma_k)_{k\in\mathbb{N}}$, for each $r\geqslant 3$,  is a classical result. For instance, first, one can show the existence of large girth graphs using the probabilistic argument of Erd\H{o}s-Sachs~\cite{ES63} (see also~\cite{Sachs63} for a recursive, on the degree $r$, construction) 
or by taking iterated $\bmod\, 2$-covers of a given graph (each iteration doubles the girth)
and then, following Biggs~\cite[Lemma 3.1]{Biggs98}, prove that the diameter of the graph with the minimal number of vertices among all graphs of a given degree $r \geqslant 2$ and of girth at least $g\geqslant 3$ is at most $g$ (giving $D=1$ for such a graph, called a `cage'). However, this does not provide concrete examples and the resulting graphs are not necessarily Cayley graphs.\medskip

The first \emph{explicit} examples for $r\geqslant 3$ are given by suitable Cayley graphs of $SL_{2}(\mathbb F_p)$, where $\mathbb F_p=\mathbb{Z}/p\mathbb{Z}$ are the integers $\mod p$ for a prime $p$. Namely, the famous explicit construction of large girth graphs by Margulis~\cite{M82} is made of $4$-regular Cayley graphs $\Gamma_p=Cay(SL_{2}(\mathbb F_p), \{A_p,B_p\})$ with 
matrices $A_p,B_p$ being images of Sanov's generators of the free group (see below for notation and Section~\ref{sec:2} for details). It satisfies $\gi \Gamma_p\geqslant C\log |\Gamma_p|$ for $C>0$ and it is dg-bounded as by Selberg's theorem~\cite{selberg}, combined with the transfer principle of Brooks~\cite{Br} and Burger~\cite{Bu}, they form an infinite expander family as $p\to \infty$. Then, $\di\Gamma_p=O(\log |\Gamma_p|)$ because every expander has logarithmic diameter. 
Alternatively, the celebrated Lubotzky-Phillips-Sarnak~\cite{LPS} graphs, $\Gamma_q=Cay(PGL_{2}(\mathbb F_q), \{S_1,S_2, \ldots, S_{p+1}\})$ as $q\to\infty$, where $p,q$ are distinct primes congruent to 1 $\mod 4$ 
with the Legendre symbol $\legendre{p}{q}=-1$ and  $S_1, \ldots, S_{p+1}$ are suitable $p+1$ matrices,
form a Ramanujan family (hence, an expander family) and they satisfy $\gi \Gamma_q\geqslant 4 \log_pq-\log_p4$ with $\vert \Gamma_q\vert=q(q^2-1)$. Therefore,   
they are $(p+1)$-regular large girth dg-bounded Cayley graphs. Observe that both of these examples are in \emph{dimension 2} (i.e., in $2\times 2$ matrices), they are of logarithmic girth, and
the expansion property is used to conclude the dg-boundedness. Moreover, 
these constructions of Margulis and Lubotzky-Phillips-Sarnak (and their slight variants, again in dimension 2) have been, up to now, the only known explicit large girth dg-bounded expanders among the Cayley graphs.\medskip

Our aim is to provide an explicit construction of 4-regular large girth dg-bounded Cayley graphs of $SL_n(\mathbb F_p)$  as $p\to\infty$, for \emph{all dimensions} $n\geqslant 2$.
Concretely, we shall give an extensive description of two-element generating sets of $SL_n(\mathbb F_p)$ that induce the required properties of the Cayley graphs.
We formulate now our main result, indicate several corollaries (among others, constructions of $2k$-regular examples for \emph{all} $k\geqslant 2$) and explain the motivation
for this work.  

\medskip

{\bf Magic matrices in dimension $n\geqslant 2$:}\medskip

Let $ A = \begin{pmatrix}
	1 & a & 0 & 0& \ldots & 0\\
	0 & 1 & a & 0& \ldots & 0\\
	0 & 0 & 1 & a & \ldots &0\\
	\vdots & & & & & \vdots\\
	&&&&\ldots&a\\
	0 & 0 & 0& 0&\ldots & 1\\
	\end{pmatrix}$ and $B = \begin{pmatrix}
	1 & 0 & 0 & \ldots & &0\\
	b & 1 & 0 & \ldots & &0\\
	0 & b & 1 &  \ldots& &0\\
	0 & 0 & b & \ldots & &0\\
	\vdots &&&&\vdots&\vdots\\
		0& 0 & 0 & \ldots & {\strut \strut b} & 1\\
	\end{pmatrix}\in SL_{n}(\mathbb{Z})$ with ${a,b\geqslant 2.}$
\medskip

\noindent We denote by $A_p$  and $B_p$ their reductions modulo a prime $p$, and by
$\langle A,B\rangle\leqslant SL_n(\mathbb Z)$ and $\langle A_p,B_p\rangle\leqslant SL_n(\mathbb F_p),$
the subgroups generated by these matrices. \medskip

\begin{Mtheorem} The matrices $A$ and $B$ satisfy the following.\medskip

{\rm I. \underline{Freeness}:} 

\begin{myitemize}
\item For $n=2$, $\langle A,B\rangle$ is free;
\item For $n=3$, $\forall l\geqslant 4$, $\langle A^l,B^l\rangle$ is free;
\item For each $n\geqslant 4$, $\forall l\geqslant 3(n-1)$, $\langle A^l,B^l\rangle$ is free.
\end{myitemize}
\smallskip

{\rm II. \underline{Generation $\mod p$}:} 

\begin{myitemize}
\item For $n=2$, $\forall p$ prime with $a,b\not\equiv 0 (\mod p)$, we have $\langle A_p, B_p\rangle=SL_2(\mathbb F_p)$;
\item For $n=3$,  $a\equiv 1 (\mod 3), b\equiv -1(\mod 3), \forall l=4^k, k\in \mathbb N$,  we have $\langle A_p^l, B_p^l\rangle=SL_3(\mathbb F_p)$ for $p=3$ and all primes $p>K$, where $K$ is a constant;
\item For each $n\geqslant 4$, $\forall q$ prime with $n\equiv 1(\mod q)$ and $a,b$ with  $a\equiv 1 (\mod q), b\equiv 1(\mod q)$,
 $\forall l\in \lbrace 1 \rbrace\cup \lbrace q^{k+1}+1\, \colon k\in \mathbb N, k\geqslant t, \hbox{ with $t\in\N$ given by } q^t\leqslant n <q^{t+1} \rbrace $, we have
 $\langle A^l_p,B^l_p\rangle=SL_n(\mathbb F_p)$ for $p=q$ and all primes $p>L$, where $L= L(n,a,b,l)$ is a constant.
\end{myitemize}
\smallskip

{\rm III. \underline{Girth and diameter}:} 

For each of the following choices of the parameters: 

\begin{myitemize}
\item if $n=2$, $l=1$, then $\forall p$ prime with $a,b\not\equiv 0 (\mod p)$;
\item if $n=3$,  $a\equiv 1 (\mod 3), b\equiv -1(\mod 3), \forall l=4^k, k\in \mathbb N$, then for $p=3$ and all primes $p>K$, where $K$ is a constant;
\item if $n\geqslant 4$ is fixed, $\forall q$ prime with $n\equiv 1(\mod q)$, and $a,b$  with $a\equiv 1 (\mod q), b\equiv 1(\mod q)$, 
$\forall l\in \lbrace q^{k+1}+1 \hbox{ if } q\neq 2 \hbox{ and } q^{k+2}+1 \hbox{ if } q=2\, \colon k\in \mathbb N, k\geqslant t, \hbox{ with $t\in \N$ given by } q^t\leqslant n <q^{t+1} \rbrace $,
then  for $p=q$ and all primes $p>L$, where $L = L(n,a,b,l)$ is a constant,
\end{myitemize}
we have 
$$\gi Cay(SL_n(\mathbb F_p), \{ A_p^l, B_p^l\})\geqslant c_n\log p\,
 \hbox{  and }\,
\di Cay(SL_n(\mathbb F_p), \{ A_p^l, B_p^l\})= O(\log p),
 $$
where $c_n=c_n(a,b,l)$ is a constant.
\medskip

In particular, $\Gamma_p^{n,l}(a,b)=Cay(SL_n(\mathbb F_p), \{ A^l_p, B^l_p\})$ as $p\to\infty$ is a large girth dg-bounded graph,
whenever $n$ and $l$ are as above. \smallskip

Moreover, it is an expander. In addition, all the above constants $K, L, c_n$ and that of $O(\log p)$ term are effective.
%, i.e. computable from the given parameters by an algorithm.\smallskip
\end{Mtheorem}
\medskip

Taking specific dimension $n$ and power $l$ of our matrices gives, for instance,
the following sequences of 4-regular large girth dg-bounded graphs:
$$
\Gamma_p^{3,4}(4,2)=Cay(SL_3(\mathbb F_p), \{ {\footnotesize \begin{pmatrix}
	1 & 4 & 0\\
	0 & 1 & 4\\
	0 & 0 & 1
	\end{pmatrix}^{\tiny 4}, \begin{pmatrix}
	1 & 0 & 0\\
	2 & 1 & 0\\
	0 & 2 & 1
	\end{pmatrix}^4\}}) \hbox{ as } p\to\infty
$$
and
$$
\Gamma_p^{4,10}(4,7)=Cay(SL_4(\mathbb F_p), \{ {\footnotesize \begin{pmatrix}
	1 & 4 & 0 & 0\\
	0 & 1 & 4 & 0\\
	0 & 0 & 1 & 4\\
	0 & 0 & 0 & 1\\
	\end{pmatrix}^{\tiny 10}, \begin{pmatrix}
	1 & 0 & 0 & 0\\
	7 & 1 & 0 & 0\\
	0 & 7 & 1 & 0\\
	0 & 0 & 7 & 1
	\end{pmatrix}^{10}\}}) \hbox{ as } p\to\infty.
$$
\medskip

The graphs produced by the Main Theorem and the proof itself provide explicit examples relevant to various subjects. For instance, in each dimension $n\geqslant 3$,
every free subgroup $\langle A^l, B^l\rangle$ from the Main Theorem is an example of a \emph{thin} matrix group.  No explicit examples of thin free subgroups of $SL_n(\mathbb Z)$ for $n\geqslant 3$ were previously known.    Also, the
4-regular graphs $\Gamma_p^{n,l}(a,b)$ have $2k$-regular counterparts $\Gamma_p^{n,l}(a,b;k)$, on the same vertex set!
Moreover, differing conceptually from the prior examples in dimension $2$ (because produced by a thin group in contrast to an arithmetic group in case $n=2$),
our expanders are concrete test graphs in the quest for large girth \emph{super-expanders}. We discuss all these corollaries in detail, together with the corresponding open questions,
in Section~\ref{sec:questions}.\medskip

A classical application of explicit  large girth graphs, and specifically logarithmic girth graphs, is in the coding theory, e.g., to the LDPC codes, as pioneered by Margulis~\cite{M82}.
Besides being a purely combinatorial challenge in a higher dimension in contrast to known results in dimension~2, with potential applications in computer science, 
our motivation to find explicitly such graphs comes from several recent results in
geometric group theory and in metric geometry. For instance, uniformly bounded degree $r\geqslant 3$ large girth dg-bounded graphs are required
in the constructions of infinite finitely generated groups with  prescribed subgraphs in their Cayley graphs, so-called `infinite monster groups'.
A foundational example is Gromov's random groups that contain infinite expander families in their Cayley graphs~\cite{GRW,ArzDel08}.
Gromov's monster groups do not coarsely embed into a Hilbert space and are counterexamples to an important conjecture in topology, the Baum-Connes conjecture
with coefficients~\cite{HigLafSka02}. 
For a more recent
result on monster groups, see~\cite{AT18} (where it is essential to have the large girth dg-bounded graph made of Cayley graphs) and references therein.
In all these constructions, the imposed large girth and dg-boundedness of graphs ensure the existence of appropriate small cancellation labelings of their edges. This in turn guarantees a suitable embedding of such a graph into the Cayley graph of the group (whose generators are labeling letters and relators are defined by the labels of closed cycles). On the other hand, some results in metric geometry use directly large girth and dg-boundedness  (no additional labeling is required).
For example, in~\cite{Ost}, the author constructs the first examples of regular large girth graphs with uniformly bounded $\ell_1$ distortion.
Namely, given a uniformly bounded degree large girth dg-bounded graph, he applies the results of~\cite{AGS}  to the $\mod 2$-cover of such a graph to conclude a uniform bound on the $\ell_1$ distortion. Graphs with uniformly bounded $\ell_1$ distortion are useful, for instance, in the theory of approximation algorithms.

%%%%%%%%%%%%%%%%%%%%%%%%%%%%%%%%%%%%%%
\begin{notation}
Throughout the text, we denote by $Cay(G,S)$ the Cayley graph of a group $G$ with respect to a generating set $S\subseteq G$.
In particular, $Cay(G,S\sqcup S^{-1})$ is the same graph even if a given $S$ is not symmetric.
\end{notation}

\begin{thank}
We thank Peter Sarnak for indicating that free subgroups from our Main Theorem are thin matrix groups, see Section~\ref{sec:questions}, 
Alex Lubotzky for pointing out that our result also yields $r$-regular logarithmic girth Cayley expander graphs, for all even integer $r\geqslant 4$, see Corollaries~\ref{cor:lub} and \ref{cor:lub2} and Nikolay Nikolov for attracting our attention to Nori's result~\cite{Nori}.  We are grateful to the anonymous referees for detailed comments and useful suggestions.
\end{thank}

%%%%%%%%%%%%%%%%%%%%%%%%%%%%%%%%%%%%%%
\section{Strategy of the proof and expansion properties}\label{sec:strategy}
Our strategy  has an advantage to guarantee the expansion properties of our graphs 
as a byproduct of previously known strong results about expanders defined through finite quotients of Zariski dense subgroups of $SL_n(\mathbb Z)$.

In detail, the proof of Main Theorem consists of four steps. For each of the choices of dimension, $n=2,3$ and $n\geqslant 4$, we prove that our matrices $A,B\in SL_n(\mathbb Z)$ satisfy the following:
\begin{itemize}
\item[(i)] $A^l$ and $B^l$ generate a free subgroup in $SL_n(\mathbb Z)$ for each $l\geqslant 1$ as in the Main Theorem ($l=1$ for $n=2$), 
\item[(ii)] $A^l_p, B^l_p$, their reduction $\mod p$, generate $SL_{n}(\mathbb{F}_{p})$ for all sufficiently large primes $p$, 
\item[(iii)] $\gi Cay(SL_n(\mathbb F_p), \{ A^l_p, B^l_p\})\geqslant c_n\log p$ for  some $c_n>0$;
\end{itemize}
In addition, we deduce:
\begin{itemize}
\item[(iv)] $Cay(SL_n(\mathbb F_p), \{ A^l_p, B^l_p\})$ as $p\to\infty$ is an expander.
\end{itemize}
Each of these assertions is new for our pair of matrices in case $n\geqslant 3$ and it seems no explicit matrices have been known with the properties (i)--(iii), and (i)--(iv), for $n\geqslant 3$.

We proceed as follows. For each of the choices of dimension, we first show properties (i) and (ii), and then use them and known results about growth of small subsets in $SL_{n}(\mathbb{F}_{p})$ to obtain a logarithmic bound on the diameter of the Cayley graphs.
We do not use the expansion properties of the involved graphs as our focus is large girth dg-bounded graphs and, a priori, such graphs need not be expanders.
The logarithmic bound on the diameter together with (iii) gives the required large girth with bounded diameter-by-girth ratio.
Although, our strategy is similar in all dimensions, dimension $n=3$ 
appears to be exceptional. Note also a smaller value of $l$ we can allow in this case in Main Theorem. 

In dimension $n=2$, we moreover have property (ii) for \emph{all} primes $p$ such that ${a,b\not\equiv 0 (\mod p)}$.
For all \emph{sufficiently large} primes $p$, (ii) can be obtained from (i) by the Matthews-Vaserstein-Weisfeiler theorem~\cite{MR735226}
as the free subgroup of $SL_2(\mathbb Z)$ is non-elementary, and hence Zariski dense. 
We do not use this deep result in our proof of (ii) for $n=2$ but present an easy direct argument for the sake of completeness.

In dimension $n\geqslant 3$,
in contrast to the two-dimensional case, assertion (i) does not yield (ii) since
in higher dimensions freeness does not imply Zariski dense, so~\cite{MR735226} does not apply immediately.
However, our computations show that assertion (ii) holds for a suitable prime $p$. 
This, combined with~\cite{W91} (see also \cite{myfav85}), implies that our free subgroup $\langle A^l,B^l\rangle$ is indeed Zariski dense, and, hence,
by~\cite{MR735226}, (ii) holds for all sufficiently large primes $p$.

The expansion properties of our sequence of Cayley graphs have not been known previously. 
Observe that generators $A^l_p, B^l_p$ of $SL_{n}(\mathbb{F}_{p})$ are \emph{not} images\footnote{For $n\geqslant 3$, $G=SL_n(\mathbb Z)$ has Kazhdan's property (T), hence any generating set $S$ and a nested family of finite index normal subgroups 
$G=G_0\trianglerighteqslant G_1 \trianglerighteqslant \cdots$ with~$\bigcap_{n=0}^{\infty}G_n=\{1\}$
naturally provide an infinite expander $\bigsqcup_{n=0}^{\infty} Cay(G/G_n, S_n),$ where $S_n$ is 
the canonical image of $S$. Such expanders are neither of large girth, nor dg-bounded.} of generators
of\   $SL_n(\mathbb Z)$.
However, as explained, using (i) for $n=2$ and (i)-(ii) for $n\geqslant 3$, the free subgroup $\langle A^l,B^l\rangle$ is Zariski dense.
It follows by \cite{MR2415383} for $n=2$ and by \cite{MR2897695} for $n\geqslant 3$ that
$Cay(SL_n(\mathbb F_p), \{ A^l_p, B^l_p\})$ as $p\to\infty$ is indeed an expander.
It appears to be the first explicit expander in dimension $n\geqslant 3$ which is large girth dg-bounded, by the Main Theorem.

%%%%%%%%%%%%%%%%%%
\section{Dimension $n=2$}\label{sec:2}
%%%%%%%%%%%%%%%%%%

The results of this section are certainly known to specialists, although
formulations available in the literature often restrict to the cases  $|a|=|b|=2$ or $|a|=|b|$. In our setting, $a$ and $b$ can differ and  we present arguments for completeness.

\subsection{Girth}
Our magic matrices  are $2$-by-$2$ matrices $A = \begin{pmatrix}
1 & a\\
0 & 1
\end{pmatrix}$ and $B = \begin{pmatrix}
1 & 0 \\
b & 1
\end{pmatrix}$. Enlarging the set of possible $a,b$, we take $a,b\in \mathbb{Z}, |a|, |b|\geqslant 2$. 
It is an easy consequence of the ping-pong lemma that the subgroup generated by these matrices $\langle A,B\rangle \,\leqslant SL_{2}(\mathbb{Z})$ is free.

\begin{lemma}[Ping-pong~{\cite[Ch.III, Prop.12.2]{myfav84}}]\label{pingpong}
	Let $G$ be a group acting on a set $X$, let $\Gamma_{1}, \Gamma_{2}$ be two subgroups of $G$, and let $\Gamma = \langle\Gamma_{1}, \Gamma_{2}\rangle$ be the subgroup of $G$ generated by $\Gamma_{1}$ and $\Gamma_{2}$. Suppose $|\Gamma_{1}|\geqslant 3$ and $|\Gamma_{2}|\geqslant 2$. Suppose there exist two disjoint non-empty subsets of $X$, say $X_{1}$ and $X_{2}$ with $$\gamma(X_{1}) \subset X_{2}\,\,\, \forall \gamma \in \Gamma_{2}, \gamma \neq 1 \text{ and } \gamma(X_{2}) \subset X_{1}\,\,\, \forall \gamma \in \Gamma_{1}, \gamma \neq 1.$$	
	Then, $\Gamma$ is isomorphic to the free-product $\Gamma_{1}*\Gamma_{2}$.
\end{lemma}

 \begin{corollary}[Sanov's theorem]\label{cor:sanov}
 	$\langle A,B\rangle \,\leqslant SL_{2}(\mathbb{Z})$ is free.
 \end{corollary}

\begin{proof}
	Consider the usual action of $SL_{2}(\mathbb{Z})$ on $X =\mathbb{Z}^{2} = \left\{\begin{pmatrix}
	x \\
	y
	\end{pmatrix} : x,y\in \Z\right\}$. Let $\Gamma_{1} = \langle A\rangle, \Gamma_{2} = \langle B\rangle $, $X_{1} = \left\lbrace \begin{pmatrix}
		x \\
	y
	\end{pmatrix} \in \mathbb{Z}^{2} : |x| > |y|\right\rbrace$ and $X_{2} = \left\lbrace \begin{pmatrix}
	x \\
	y
	\end{pmatrix} \in \mathbb{Z}^{2} : |y| > |x|\right\rbrace$ be two disjoint non-empty subsets of $X$.\\
	
	Let $\gamma_{1}\in \Gamma_{1}$ and $\gamma_{2}\in \Gamma_{2}$. Then $\gamma_{1} = \begin{pmatrix}
	1 & k_{1}a\\
	0 & 1
	\end{pmatrix}$ and $\gamma_{2} = \begin{pmatrix}
	1 & 0\\
	k_{2}b & 1
	\end{pmatrix}$ for some $k_{1}, k_{2}\in \mathbb{Z}\backslash \lbrace 0\rbrace$. Clearly,	$\gamma_{1}(X_{2}) = \left\lbrace\begin{pmatrix}
	x +k_{1}ay\\
	y
		\end{pmatrix} : |y| > |x|  \right\rbrace $ and $|x|< |y| \Rightarrow |x+k_{1}ay| > |y|$ since $|k_{1}a|\geqslant 2$ (the assumption $|a|\geqslant 2$ is essential). Hence, $\gamma_{1}(X_{2}) \subset X_{1}\,\,\forall \gamma_{1}\in \Gamma_{1} \backslash \lbrace 1 \rbrace $. Similarly, $\gamma_{2}(X_{1})\subset X_{2} \,\,\ \forall \gamma_{2}\in \Gamma_{2}\backslash \lbrace 1 \rbrace $ (with the use of assumption $|b|\geqslant 2$). The ping-pong lemma applies and we are done as 
		$\Gamma_1\cong\Gamma_2\cong\Z$.
\end{proof}

	Note that	for $a=b=1$, the subgroup $\langle A,B\rangle$ is not free. For $|a|=|b|=2$, it is of finite index in $SL_{2}(\mathbb{Z})$, while for $|a| = |b|>2$, it is of infinite index in $SL_{2}(\mathbb{Z})$. \medskip 

The main result of this subsection is Proposition \ref{mainpropsec2}. 
This fact is originally due to Margulis for $a=b=2$~\cite{M82}, cf.~\cite[Appendix A]{DSV}. 
We give our computation, which is a bit more explicit in view of its generalization to higher dimensions.

\begin{proposition}\label{mainpropsec2}
	Let $p$ be a prime. Then $\gi Cay(\langle A_{p},B_{p}\rangle,\{ A_{p}, B_{p}\})\geqslant C\log p$ for a constant $C=C(a,b)>0$.
\end{proposition}

\begin{proof}
The girth is the length of the shortest non-trivial cycle in $Cay(\langle A_{p},B_{p}\rangle,\{ A_{p}, B_{p}\})$. Hence, it is the minimum length of the non-trivial word in the generators $\lbrace A_{p}^{\pm 1},B_{p}^{\pm 1} \rbrace $ which represents the identity of $\langle A_{p},B_{p}\rangle$. 

Starting at the identity of $\langle A_{p},B_{p}\rangle$, we consider all non-trivial paths in the Cayley graph which return to the identity. 
We aim to show that the length of all such paths in $Cay(\langle A_{p},B_{p}\rangle,\{ A_{p}, B_{p}\})$ is at least $C\log p$ for some $C>0$, whenever $a,b$ are integers such that $a,b\not\equiv 1 (\mod p)$ and  $a,b\not\equiv 0 (\mod p)$. These assumptions on $p$ are not restrictive as $a,b$ are fixed, so there are finitely many excluded primes and the values of 
the corresponding girths are bounded above by a constant. 

The proof is an explicit analysis of the growth of the products 
$$\displaystyle\prod_{1\leqslant i\leqslant k} (A^{l_{i}}B^{m_{i}}) \in SL_{2}(\mathbb Z), \,\,  i,k\in \mathbb{N},\,\, l_{i}, m_{i}\in \mathbb{Z}\backslash \lbrace 0 \rbrace,$$ where $$A=\begin{pmatrix}
1 & a\\
0 & 1
\end{pmatrix}, B = \begin{pmatrix}
1 & 0\\
b & 1\\
\end{pmatrix}, \,\, a,b\not\equiv 1 (\mod p) \hbox{ and } a,b\not\equiv 0 (\mod p).$$ We have $A^{l_{i}} = \begin{pmatrix}
1 & l_{i}a\\
0 & 1
\end{pmatrix}, B^{m_{i}} = \begin{pmatrix}
1 & 0 \\
m_{i}b & 1
\end{pmatrix} \,\, \forall l_{i},m_{i} \in \mathbb{Z}\backslash \lbrace 0 \rbrace$, which gives \\
\\
$$A^{l_{1}}B^{m_{1}} = \begin{pmatrix}
l_{1}m_{1}ab + 1 & l_{1}a\\
m_{1}b & 1\\
\end{pmatrix},$$
$$A^{l_{1}}B^{m_{1}}A^{l_{2}}B^{m_{2}} = \begin{pmatrix}
(l_{1}m_{1}l_{2}m_{2})(ab)^{2} + (l_{1}m_{1} + l_{1}m_{2} + l_{2}m_{2})ab + 1  &  (l_{1}m_{1}l_{2})a^{2}b + (l_{1}+l_{2})a\\
(m_{1}l_{2}m_{2})ab^{2} + (m_{1}+m_{2})b & (m_{1}l_{2})ab + 1\\
\end{pmatrix},$$
and in general if we have the product of $2k$ terms then each entry of the matrix is a polynomial in $a,b$. Let us denote these polynomials by $P_{11}(a,b), P_{12}(a,b),P_{21}(a,b), P_{22}(a,b)$ such that 
$$ \displaystyle\prod_{1\leqslant i\leqslant k} A^{l_{i}}B^{m_{i}} = \begin{pmatrix}
P_{11}(a,b) & P_{12}(a,b)\\
P_{21}(a,b) & P_{22}(a,b)\\
\end{pmatrix}, \,\, l_{i}, m_{i}\in \mathbb{Z}\backslash \lbrace 0 \rbrace \,\forall i$$
then, by induction on $k$, one can obtain the explicit forms of the polynomials $P_{ij}(a,b), 1\leqslant i,j\leqslant 2$.
For instance, $P_{11}$ has the form, \\
\\
$P_{11}(a,b) = 1 + a_{2}(ab) + a_{4}(ab)^{2} + \ldots + a_{2k}(ab)^{k},$ where, for $1\leqslant r \leqslant k$, we have $$a_{2r}= \sum_{1\leqslant j_{r}\leqslant i_{r}< \ldots<j_{2}\leqslant i_{2}< j_{1}\leqslant i_{1}\leqslant k} m_{i_{1}}l_{j_{1}}m_{i_{2}}l_{j_{2}}\cdots m_{i_{r}}l_{j_{r}}.$$
\\
%\\	
%$P_{12}(a,b) = b_{1}a + b_{3}a^{2}b + \ldots + b_{2k-1}a^{k}b^{k-1},$\\
%where 
%$$b_{1} = \sum_{1\leqslant j_{1}\leqslant k}l_{j_{1}}, b_{3} = \sum_{1\leqslant j_{2}\leqslant i_{1}<j_{1}\leqslant k} l_{j_{1}}m_{i_{1}}l_{j_{2}}$$ and for $1\leqslant r \leqslant k$ we have $$b_{2r-1} = \sum_{1\leqslant j_{r}\leqslant i_{r-1}<j_{r-1}\leqslant...<j_{2}\leqslant i_{1}<j_{1}\leqslant k}l_{j_{1}}m_{i_{1}}l_{j_{2}}\cdots l_{j_{r-1}}m_{i_{r-1}}l_{j_{r}}.$$
%$P_{21}(a,b) = c_{1}b + c_{3}ab^{2} + \ldots + c_{2k-1}a^{k-1}b^k,$\\
%where 
%$$c_{1} = \sum_{1\leqslant i_{1}\leqslant k}m_{i_{1}}, c_{3} = \sum_{1\leqslant i_{2}< j_{1}\leqslant i_{1}\leqslant k}m_{i_{1}}l_{j_{1}}m_{i_{2}} $$ and for $1\leqslant r \leqslant k$ we have $$c_{2r-1} = \sum_{1\leqslant i_{r}<j_{r-1} \leqslant \ldots \leqslant i_{2}<j_{1}\leqslant i_{1}\leqslant k}m_{i_{1}}l_{j_{1}}m_{i_{2}}\cdots l_{j_{r-1}}m_{i_{r}}.$$
%Finally, 
%$P_{22}(a,b) = 1 + d_{2}(ab) + d_{4}(ab)^{2} + \ldots + d_{2k-2}(ab)^{k-1}$, where\\
%$$d_{2} = \sum_{1\leqslant i_{1} <j_{1}\leqslant k} l_{j_{1}}m_{i_{1}}, \,\, d_{4} = \sum_{1\leqslant i_{2}<j_{2}\leqslant i_{1}< j_{1}\leqslant k} l_{j_{1}}m_{i_{1}}l_{j_{2}}m_{i_{2}}$$ and for $1\leqslant r \leqslant k$ we have $$d_{2r}= \sum_{1\leqslant i_{r}< j_{r}\leqslant ...\leqslant i_{2}<j_{2}\leqslant i_{1}<j_{1}\leqslant k} l_{j_{1}}m_{i_{1}}l_{j_{2}}m_{i_{2}}\cdots l_{j_{r}}m_{i_{r}}.$$
Since $A,B$ generate a free group, $\displaystyle\prod_{1\leqslant i\leqslant k} A^{l_{i}}B^{m_{i}} \neq \rm{id}$ in $SL_{2}(\mathbb{Z})$ and so starting from the identity a necessary condition for the path along this product to reach the identity in $Cay(\langle A_{p},B_{p}\rangle, \{A_p,B_p\})$ is $|P_{ij}(a,b)| > p$, for some $1\leqslant i,j\leqslant 2$. Let $k$ denote the length of such a path. Then, $$|P_{ij}(a,b)| \leqslant M^{k}|(ab+1)|^{k},$$ where $M = \max \lbrace |l_{i}|,|m_{i}| : 1\leqslant i\leqslant k \rbrace.$ This yields a constant $C = C(a,b)> 0$ such that 
$$k > C\log p.$$  
From the above, it is also clear that the same holds when $l_{1} = 0$ or $m_{k} = 0$.
Since Cayley graphs are vertex transitive and we have shown that starting from the identity it takes the product of at least $C\log p$ matrices of the form $A_{p}^{\pm 1},B_{p}^{\pm 1}$ to reach the identity non-trivially, we conclude that the girth of $Cay(\langle A_{p},B_{p}\rangle, \{A_p,B_p\})$ is at least $C \log p$.
\end{proof}

%%%%%%%%%%%%%%%%%%%%%%%
\subsection{Diameter} 
The diameter problem for the finite (simple) linear groups has been studied extensively and there exists a vast literature on the subject.
We follow our strategy from Section~\ref{sec:strategy} and make sure that our elements $A_p$ and $B_p$ generate the entire $SL_{2}(\mathbb{F}_{p})$. In dimension $2$, this fact is easy.
\begin{lemma}\label{lemsec2}
	$\langle A_p, B_p\rangle =SL_{2}(\mathbb{F}_{p})$ for all primes $p$.
\end{lemma}

\begin{proof}
	Fix a prime $p$ and note that $A^{l} \equiv \begin{pmatrix}
	1 & 1 \\
	0 & 1
	\end{pmatrix} (\mod p)$ and $B^{m} \equiv \begin{pmatrix}
	1 & 0\\
	1 & 1
	\end{pmatrix} (\mod  p)$ for some $1\leqslant l,m\leqslant p$. 
	Since the $\mod p$ reduction $SL_{2}(\mathbb Z)\twoheadrightarrow SL_{2}(\mathbb{F}_{p})$ is surjective and 
	$\begin{pmatrix}
	1 & 1 \\
	0 & 1
	\end{pmatrix}$ and
	$\begin{pmatrix}
	1 & 0\\
	1 & 1
	\end{pmatrix}$
	generate $SL_{2}(\mathbb Z)$, then their $\mod p$ reduction generate $SL_{2}(\mathbb{F}_{p})$.
	It follows that  $\langle A_p,B_p\rangle=SL_{2}(\mathbb{F}_{p})$, for all primes $p$.
\end{proof}

In contrast, it is a highly non-trivial fact that the diameter of $Cay(SL_{2}(\mathbb{F}_{p}), \lbrace A_{p}, B_{p}\rbrace)$ is $O(\log p)$. 
As alluded to in the introduction, one way is to use the expansion properties of the sequence as $p\to\infty$.
Another way is to apply the circle method to show that any element of $SL_{2}(\mathbb{F}_{p})$ lifts to an element of $SL_{2}(\mathbb{Z})$ having word representation of $O(\log p)$~\cite{LPS} (see also~\cite{Larsen} for an efficient algorithm, although it gives $O(\log p \log \log p)$ only).  We take yet an alternative way 
and use an aspect of the seminal work of Helfgott \cite{myfav22}.
\begin{theorem}[Helfgott \cite{myfav22}]
	Let $p$ be a prime. Let $S$ be any generating set of $SL_{2}(\mathbb{F}_p)$.
	Then the Cayley graph $Cay(SL_{2}(\mathbb{F}_p), S)$ has diameter $O((\log p)^{c})$, where $c$ and the implied constant
	are absolute.
\end{theorem}
To establish this  theorem Helfgott showed the following result.
\begin{proposition}[Key proposition \cite{myfav22}]\label{keyprop2}
	Let $p$ be a prime. Let $S$ be a subset of $SL_{2}(\mathbb{F}_p)$ not contained
	in any proper subgroup. 
	\begin{enumerate}
		\item 	Assume that $|S| < p^{3-\delta}$ for some fixed $\delta > 0$. Then
		$$|S^{3}| > c|S|^{1+\epsilon}$$ where $c > 0$ and $\epsilon > 0$ depend only on $\delta$.	
		\item Assume that $|S| > p^{\delta}$ for some fixed $\delta > 0$. Then there is an integer $k > 0$,
		depending only on $\delta$, such that every element of $SL_{2}(\mathbb F_p)$ can be expressed as a
		product of at most $k$ elements of $S\sqcup S^{-1}$.
	\end{enumerate}

\end{proposition}

We now obtain the required upper bound on the diameter of our graphs.
\begin{lemma}\label{mainlem2}
	The diameter of $Cay(SL_{2}(\mathbb{F}_{p}), \lbrace A_{p}, B_{p}\rbrace)$ is $O(\log p)$, where the implied constant depends on $a$ and $b$.
\end{lemma}
\begin{proof}
	By Lemma \ref{lemsec2},  the group $\langle A,B\rangle $ surjects onto $SL_{2}(\mathbb{F}_{p})$. We know that $\langle A,B\rangle $ is a free subgroup in $SL_{2}(\mathbb{Z})$ and the girth of $Cay(SL_{2}(\mathbb{F}_{p}), \lbrace A_{p}, B_{p}\rbrace)$ is at least $c\log p$ for some constant $c=c(a,b)>0$. 
	Therefore, denoting $S=\lbrace A_{p}^{\pm 1}, B_{p}^{\pm 1}\rbrace$, we have  $$|S^{\frac{c}{6}\log p}| \geqslant 3^{\frac{c}{6}\log p}.$$
	Choosing $S' = S^{\frac{c}{6}\log p}$, we find ourselves in $(2)$ of Proposition \ref{keyprop2} (with $\delta < \frac{c}{6}\log 3$),
	and hence $(S')^{k} = SL_{2}(\mathbb{F}_{p}) \,\,\forall p$, where $k$ depends only on $\delta$. Therefore, $S^{kc\log p} = SL_{2}(\mathbb{F}_{p})$, which means that the diameter  of $Cay(SL_{2}(\mathbb{F}_{p}),\lbrace A_{p}, B_{p}\rbrace)$ is $O(\log p)$.
\end{proof}

	Since we have a generating set which is free in $SL_{2}(\mathbb{Z})$, the growth of balls in $SL_{2}(\mathbb{F}_{p})$  is fast at the beginning (up to $\gi $ scale). Therefore, we only used part $(2)$ of Proposition \ref{keyprop2} in the preceding proof. This estimate on the diameter (with the same argument) also appeared, 
	for example, in~\cite[Corollary~6.3]{myfav22}.

\begin{corollary}\label{cor:dg2}
Let $a,b\in \mathbb{Z}\backslash \lbrace 0,1\rbrace$.
	The diameter-by-girth ratio of the sequence of Cayley graphs of $G = SL_{2}(\mathbb{F}_{p})$ with respect to $S = \lbrace A_{p}^{\pm 1}, B_{p}^{\pm 1}\rbrace$, as $p\to\infty$, where $A = \begin{pmatrix}
	1 & a\\
	0 & 1
	\end{pmatrix}, B =  \begin{pmatrix}
	1 & 0\\
	b & 1
	\end{pmatrix}$ is bounded by a constant.
\end{corollary}
\begin{proof}
	By Proposition \ref{mainpropsec2} and Lemma \ref{mainlem2}, there exist constants $K_{1} = K_{1}(a,b)>0$ and $K_{2} = K_{2}(a,b)>0$ such that $\gi Cay(G,S) > K_{1}\log p$ and $\di Cay(G,S) < K_{2}\log p$, respectively. Hence, $\frac{\di Cay(G,S)}{\gi Cay(G,S)} < \frac{K_{2}}{K_{1}}$.
\end{proof}

%%%%%%%%%%%%%%%%%%%%%%
\section{Dimension $n=3$}

\subsection{Girth}
For $n=3$ the situation is more complicated. The primary difficulty is to find suitable candidates for our free subgroup in $SL_{3}(\mathbb Z)$. 
The existence of such free subgroups is a well-known fact but explicit examples seem not to be present in the literature.
Indeed, apart from a few special cases the ping-pong lemma, Lemma~\ref{pingpong}, is the universal way one can establish that two elements generate a non-abelian free subgroup. The challenge with the ping-pong lemma is that it is a non-trivial problem to find an explicit description of disjoint non-empty subspaces $X_{1}$ and $X_{2}$ such that $\gamma_{1}(X_{2})\subset X_{1}$ and $\gamma_{2}(X_{1})\subset X_{2}$. In higher dimensions, each of  the sets $X_{1}$ and $X_{2}$ might be a union of smaller subsets. We get around this difficulty by 
a direct computation on growth of some products.  We show that fourth powers of our $A$ and $B$  
generate a free group inside $SL_{3}(\mathbb{Z})$. Then we use them to produce a new large girth dg-bounded sequence of finite Cayley graphs.

\begin{proposition}\label{mainpropsec3} Fix $a,b\in \mathbb{N}, a,b\geqslant 2.$
	Let $A,B\in SL_{3}(\mathbb{Z})$ be such that $$A = \begin{pmatrix}
	1 & a & 0\\
	0 & 1 & a\\
	0 & 0 & 1
	\end{pmatrix}, \,B = \begin{pmatrix}
	1 & 0 & 0\\
	b & 1 & 0\\
	0 & b & 1
	\end{pmatrix}
	$$
	
	Then, $\langle A^{4},B^{4}\rangle$ is a free subgroup of $SL_{3}(\mathbb{Z})$.
\end{proposition}

\begin{proof}
	Let $X = A^4 = \begin{pmatrix}
	1 & 4a & 6a^{2} \\
	0 & 1 & 4a \\
	0 & 0 & 1\\
	\end{pmatrix}$ and $Y = B^{4} = \begin{pmatrix}
	1 & 0 & 0 \\
	4b & 1 & 0 \\
	6b^{2} & 4b & 1\\
	\end{pmatrix}$. 
	We claim that $\langle X,Y\rangle$ is a free subgroup of $SL_{3}(\mathbb{Z})$. 
	We study products of the form $X^{r_{1}}Y^{s_{1}}\cdots X^{r_{k}}Y^{s_{k}}$ for any $r_{i},s_{i} \in \mathbb{Z}\backslash \lbrace 0\rbrace, k\in \mathbb{N}$.  The crucial step is to show that 	$$\big(X^{r_{1}}Y^{s_{1}}\big)\big(X^{r_{2}}Y^{s_{2}}\big)\cdots \big(X^{r_{k}}Y^{s_{k}}\big) \neq {\rm id} \hbox{ in } SL_3(\mathbb Z),\,\, \forall r_{i},s_{i} \in \mathbb{Z}\backslash \lbrace 0\rbrace, k\in \mathbb{N}.$$
	
	Clearly, $$X^{r_{i}}Y^{s_{i}} = \begin{pmatrix}
	1 & 4ar_{i} & \frac{4r_{i}a^{2}(4r_{i}-1)}{2}\\
	0 & 1 & 4ar_{i}\\
	0 & 0 & 1\\ 
	\end{pmatrix} \times \begin{pmatrix}
	1 & 0 & 0\\
	4bs_{i} & 1 & 0\\
	\frac{4s_{i}b^{2}(4s_{i}-1)}{2} & 4bs_{i} & 1 \\ 
	\end{pmatrix} $$
	
	$$= 4abr_{i}s_{i} \begin{pmatrix}
	ab(4r_{i}-1)(4s_{i} - 1) + 4 + \frac{1}{4abr_{i}s_{i}} & 2a(4r_{i}-1) + \frac{1}{bs_{i}} & \frac{a(4r_{i}-1)}{2bs_{i}}\\
	2b(4s_{i}-1) + \frac{1}{ar_{i}} & 4 + \frac{1}{4abr_{i}s_{i}} & \frac{1}{bs_{i}}\\
	\frac{b(4s_{i}-1)}{2ar_{i}} & \frac{1}{ar_{i}} & \frac{1}{4abr_{i}s_{i}}
	\end{pmatrix} =
	4r_{i}s_{i}(4r_{i}-1)(4s_{i} - 1)\times $$
	$$\begin{pmatrix}
	x^{2} + \frac{4x}{(4r_{i}-1)(4s_{i} - 1)} + \frac{1}{4r_{i}s_{i}(4r_{i}-1)(4s_{i} - 1)} & \frac{2xa}{(4s_{i} - 1)} + \frac{a}{s_{i}(4r_{i}-1)(4s_{i} - 1)} & \frac{a^{2}}{2s_{i}(4s_{i} - 1)}\\
	
	\frac{2xb}{(4r_{i}-1)}+ \frac{b}{r_{i}(4r_{i}-1)(4s_{i} - 1)} & \frac{4x}{(4r_{i}-1)(4s_{i} - 1)}+ \frac{1}{4r_{i}s_{i}(4r_{i}-1)(4s_{i} - 1)} & \frac{a}{s_{i}(4r_{i}-1)(4s_{i} - 1)}\\
	
	\frac{b^{2}}{2r_{i}(4r_{i}-1)} & \frac{b}{r_{i}(4r_{i}-1)(4s_{i} - 1)} & \frac{1}{4r_{i}s_{i}(4r_{i}-1)(4s_{i} - 1)}
	\end{pmatrix},$$
	where $x=ab\geqslant 4$. Thus, the above expression is equal to 
	$$ 4r_{i}s_{i}(4r_{i}-1)(4s_{i} - 1)N_{i},$$ where 
	$$N_{i} = \begin{pmatrix}
	x^{2} + \frac{4x}{(4r_{i}-1)(4s_{i} - 1)} + \frac{1}{4r_{i}s_{i}(4r_{i}-1)(4s_{i} - 1)} & \frac{2xa}{(4s_{i} - 1)} + \frac{a}{s_{i}(4r_{i}-1)(4s_{i} - 1)} & \frac{a^{2}}{2s_{i}(4s_{i} - 1)}\\
	
	\frac{2xb}{(4r_{i}-1)}+ \frac{b}{r_{i}(4r_{i}-1)(4s_{i} - 1)} & \frac{4x}{(4r_{i}-1)(4s_{i} - 1)}+ \frac{1}{4r_{i}s_{i}(4r_{i}-1)(4s_{i} - 1)} & \frac{a}{s_{i}(4r_{i}-1)(4s_{i} - 1)}\\
	
	\frac{b^{2}}{2r_{i}(4r_{i}-1)} & \frac{b}{r_{i}(4r_{i}-1)(4s_{i} - 1)} & \frac{1}{4r_{i}s_{i}(4r_{i}-1)(4s_{i} - 1)}
	\end{pmatrix} \forall i\in \mathbb{N}$$ denotes the normalised form of the product $X^{r_{i}}Y^{s_{i}}.$
	
	It is easy to check that denoting $N_{i} = \begin{pmatrix}
	a_{11_{}} & a_{12_{}} & a_{13_{}}\\
	a_{21_{}} & a_{22_{}} & a_{23_{}}\\
	a_{31_{}} & a_{32_{}} & a_{33_{}}\\
	\end{pmatrix},$ we have the (anti-)lexicographic ordering among the elements from left to right and from top to bottom in the sense that $|a_{11_{}}|> |a_{12_{}}| > |a_{13_{}}|>|a_{23_{}}| > |a_{33_{}}|, |a_{11_{}}|> |a_{12_{}}| > |a_{22_{}}| > |a_{23_{}}| > |a_{33_{}}|$, etc. In fact, we have stronger inequalities on the bounds of the values taken by $a_{ij}, \forall 1\leqslant i,j\leqslant 3$. We have: 
	\begin{enumerate}
		\item $ x^{2}+\frac{x}{2}> a_{11_{}}>x^2-\frac{x}{2}$
		\item $\frac{2}{3}xa+\frac{a}{9}\geqslant |a_{12_{}}|$
		\item $\frac{2}{3}xa+\frac{x}{9a}\geqslant |a_{21_{}}|$
		\item $ \frac{a^{2}}{6}> a_{13_{}}> 0$
		\item $\frac{a^{2}}{6}>a_{31_{}}>0$
		\item $\frac{4x}{9} + \frac{1}{9}>|a_{22_{}}|$
		\item $\frac{a}{8}>|a_{23_{}}|>0$
		\item $\frac{b}{8}>|a_{32_{}}|>0$
		\item $\frac{1}{4}>a_{33_{}}>0$
	\end{enumerate}
	
	If $N = N_{1}N_{2}\cdots N_{k}$ has the first coefficient  $>1$, then it will imply $\big(X^{r_{1}}Y^{s_{1}}\big)\big(X^{r_{2}}Y^{s_{2}}\big)\cdots\big(X^{r_{k}}Y^{s_{k}}\big) \neq {\rm id} \hbox{ in } SL_3(\mathbb Z), \,\, \forall r_{i},s_{i} \in \mathbb{Z}\backslash \lbrace 0\rbrace, k\in \mathbb{N}.$\\ 
	
	Let $Z = \begin{pmatrix}
	A & B_{1} & C_{1}\\
	B_{2} & D & E_{1}\\
	C_{2} & E_{2} & F
	
	\end{pmatrix} \in SL_{3}(\mathbb{Z})$ be such that $A > |B_{1}| + |C_{1}| + 1 $. 
	\begin{claim}\label{mainclaim3}
		$Z\cdot N_{1}\cdot N_{2}\cdots N_{k}$ has the same form as $Z$.
	\end{claim}

	\begin{proof}[Proof of claim]
Let $Z'=Z\cdot N_i$ for some $i\in\mathbb N.$ First, we check that $Z'$ has the same form as $Z$.		
		$$Z' = \begin{pmatrix}
		A' & B'_{1} & C'_{1}\\
		B'_{2} & D' & E'_{1}\\
		C'_{2} & E'_{2} & F'
		\end{pmatrix} =  \begin{pmatrix}
		A & B_{1} & C_{1}\\
		B_{2} & D & E_{1}\\
		C_{2} & E_{2} & F
		
		\end{pmatrix} \times  \begin{pmatrix}
		a_{11_{}} & a_{12_{}} & a_{13_{}}\\
		a_{21_{}} & a_{22_{}} & a_{23_{}}\\
		a_{31_{}} & a_{32_{}} & a_{33_{}}\\
		\end{pmatrix}$$ 
		By the inequalities on the $ a_{rs_{}}, 1 \leqslant r,s\leqslant 3$, we know that 
		\begin{enumerate}
			\item $A' = Aa_{11_{}} + B_{1}a_{21_{}} + C_{1}a_{31_{}} > (x^{2}-\frac{x}{2})A - (\frac{2}{3}xa+\frac{x}{9a})|B_{1}| - \frac{a^{2}}{6}|C_{1}|$
			\item $B'_{1} = Aa_{12_{}}+ B_{1}a_{22_{}} + C_{1}a_{32_{}} < (\frac{2}{3}xa+\frac{a}{9})A + (\frac{4x}{9} + \frac{1}{9})|B_{1}| + \frac{b}{8}|C_{1}|$
			\item $C'_{1} = Aa_{13_{}} + B_{1}a_{23_{}} + C_{1}a_{33_{}} < \frac{a^{2}}{6}A + \frac{a}{8}|B_{1}| + \frac{1}{4}|C_{1}|$ 
		\end{enumerate}
		
		We would like to show that $A' > |B'_{1}| + |C'_{1}| + 1$ under the assumption that 
		$A > |B_{1}| + |C_{1}| + 1.$\\
		Substituting the above inequalities, we get that if we can show
		\begin{equation}\label{eqn}\begin{split}
		\left(x^{2}-\frac{x}{2}\right)A - \left(\frac{2}{3}xa+\frac{x}{9a}\right)|B_{1}| - \frac{a^{2}}{6}|C_{1}| > \left(\frac{2}{3}xa+\frac{a}{9}\right)A + 
		\left(\frac{4x}{9} + \frac{1}{9}\right)|B_{1}| \\ + \frac{b}{8}|C_{1}| + \frac{a^{2}}{6}A + \frac{a}{8}|B_{1}| \- + \frac{1}{4}|C_{1}| + 1,
		\end{split}
		\end{equation}
		then we are done. We know that $x=ab \geqslant 2a$. If $x > 2a$ (equivalently $b>2$), then rearranging and simplifying the above expression we see that if we can show  
		$$(x^{2} - 1 -xa)A > xa|B_{1}| + xa|C_{1}| + 1,$$ 
		then we are done. Indeed, the above holds under the assumption $A > |B_{1}| + |C_{1}| + 1$ (if $x\geqslant 2a+1$). 
		For $x =2a$, a direct substitution in (\ref{eqn}) shows that it holds under this assumption.
	
	The claim follows by induction on $k$, using the above assertion on $Z'$ twice. Indeed, the base of the induction is the above considerations for $Z\cdot N_1$. 
	Then if $Z\cdot N_1\cdot N_2 \cdots N_{i-1}$ has the same form as $Z$,
	this assertion gives that $Z\cdot N_1\cdot N_2\cdots N_i$ has the same form as $Z$ as well.
		
	\end{proof}
	
	Thus, $N= N_{1}\cdots N_{k} $ cannot be identity in $SL_3(\mathbb Z)$ for any $k\in \mathbb{N}$. Also, $N$ has the same form as $Z$. 
	Since $X^r$ is not of this form, it follows that $N \neq X^{r}$ for any $r\in \mathbb{Z}$. Therefore, products of the form $N\cdot X^{r}$ cannot be identity either. The case of $Y^{s}\cdot N \neq {\rm id}$ is clear from the fact that we consider reduced words in $X^{\pm 1}$ and $Y^{\pm 1}$. So, $Y^{s}\cdot N= {\rm id} \Longrightarrow NY^{s} = {\rm id}$ which either has the same form as above or reduces to a power of $X$ or of $Y$. Both of these elements are of infinite order so their powers cannot be identity. The remaining case of products of form $Y^s\cdot N\cdot X^r$ reduces to the already considered.
	This concludes the fact that $\langle X,Y\rangle \leqslant SL_{3}(\mathbb{Z})$ is a free subgroup.
\end{proof}

Our proof of Proposition \ref{mainpropsec3} actually gives the following stronger statement.
\begin{theorem}\label{mainthmsec3}
		Let $A,B\in SL_{3}(\mathbb{Z})$ be such that $$A = \begin{pmatrix}
	1 & a & 0\\
	0 & 1 & a\\
	0 & 0 & 1
	\end{pmatrix}, \,B = \begin{pmatrix}
	1 & 0 & 0\\
	b & 1 & 0\\
	0 & b & 1
	\end{pmatrix}, \,\, a,b\geqslant 2.$$
	
	Then $\langle A^{l},B^{l}\rangle, \,\, \forall l\geqslant 4$, is a free subgroup of $SL_{3}(\mathbb{Z})$.
\end{theorem}

\begin{theorem}\label{thmgirthsec3}
	Fix $a,b\geqslant 2$ and let $p$ be a prime. Then, there exists a constant $C=C(a,b)>0$, such that 
	$\gi Cay(\langle A^{4}_{p}, B^{4}_{p}\rangle, \{A^{4}_{p}, B^{4}_{p}\} )\geqslant C\log p$.
\end{theorem}

\begin{proof}
		Let again $X = A^4 = \begin{pmatrix}
	1 & 4a & 6a^{2} \\
	0 & 1 & 4a \\
	0 & 0 & 1\\
	\end{pmatrix}$ and $Y = B^{4} = \begin{pmatrix}
	1 & 0 & 0 \\
	4b & 1 & 0 \\
	6b^{2} & 4b & 1\\
	\end{pmatrix}$. \\
	Using the previous expression for products  
	$$X^{r_{i}}Y^{s_{i}} = \begin{pmatrix}
	1 & 4ar_{i} & \frac{4r_{i}a^{2}(4r_{i}-1)}{2}\\
	0 & 1 & 4ar_{i}\\
	0 & 0 & 1\\ 
	\end{pmatrix} \times \begin{pmatrix}
	1 & 0 & 0\\
	4bs_{i} & 1 & 0\\
	\frac{4s_{i}b^{2}(4s_{i}-1)}{2} & 4bs_{i} & 1 \\ 
	\end{pmatrix} $$
	
	$$= 4abr_{i}s_{i} \begin{pmatrix}
	ab(4r_{i}-1)(4s_{i} - 1) + 4 + \frac{1}{4abr_{i}s_{i}} & 2a(4r_{i}-1) + \frac{1}{bs_{i}} & \frac{a(4r_{i}-1)}{2bs_{i}}\\
	2b(4s_{i}-1) + \frac{1}{ar_{i}} & 4 + \frac{1}{4abr_{i}s_{i}} & \frac{1}{bs_{i}}\\
	\frac{b(4s_{i}-1)}{2ar_{i}} & \frac{1}{ar_{i}} & \frac{1}{4abr_{i}s_{i}}
	\end{pmatrix} =
	4r_{i}s_{i}(4r_{i}-1)(4s_{i} - 1)\times $$
	$$\begin{pmatrix}
	x^{2} + \frac{4x}{(4r_{i}-1)(4s_{i} - 1)} + \frac{1}{4r_{i}s_{i}(4r_{i}-1)(4s_{i} - 1)} & \frac{2xa}{(4s_{i} - 1)} + \frac{a}{s_{i}(4r_{i}-1)(4s_{i} - 1)} & \frac{a^{2}}{2s_{i}(4s_{i} - 1)}\\
	
	\frac{2xb}{(4r_{i}-1)}+ \frac{b}{r_{i}(4r_{i}-1)(4s_{i} - 1)} & \frac{4x}{(4r_{i}-1)(4s_{i} - 1)}+ \frac{1}{4r_{i}s_{i}(4r_{i}-1)(4s_{i} - 1)} & \frac{a}{s_{i}(4r_{i}-1)(4s_{i} - 1)}\\
	
	\frac{b^{2}}{2r_{i}(4r_{i}-1)} & \frac{b}{r_{i}(4r_{i}-1)(4s_{i} - 1)} & \frac{1}{4r_{i}s_{i}(4r_{i}-1)(4s_{i} - 1)}
	\end{pmatrix}$$
	$$ = 4r_{i}s_{i}(4r_{i}-1)(4s_{i} - 1)N_{i},$$ where 
	$x=ab\geqslant 4$ and $ \forall i\in \mathbb{N},$
	$$N_{i} = \begin{pmatrix}
	x^{2} + \frac{4x}{(4r_{i}-1)(4s_{i} - 1)} + \frac{1}{4r_{i}s_{i}(4r_{i}-1)(4s_{i} - 1)} & \frac{2xa}{(4s_{i} - 1)} + \frac{a}{s_{i}(4r_{i}-1)(4s_{i} - 1)} & \frac{a^{2}}{2s_{i}(4s_{i} - 1)}\\
	
	\frac{2xb}{(4r_{i}-1)}+ \frac{b}{r_{i}(4r_{i}-1)(4s_{i} - 1)} & \frac{4x}{(4r_{i}-1)(4s_{i} - 1)}+ \frac{1}{4r_{i}s_{i}(4r_{i}-1)(4s_{i} - 1)} & \frac{a}{s_{i}(4r_{i}-1)(4s_{i} - 1)}\\
	
	\frac{b^{2}}{2r_{i}(4r_{i}-1)} & \frac{b}{r_{i}(4r_{i}-1)(4s_{i} - 1)} & \frac{1}{4r_{i}s_{i}(4r_{i}-1)(4s_{i} - 1)}
	\end{pmatrix}.$$
		
In general, $$\displaystyle\prod_{1\leqslant i\leqslant k}X^{r_{i}}Y^{s_{i}} = 4^{k}\prod_{1\leqslant i \leqslant k}r_{i}s_{i}(4r_{i}-1)(4s_{i}-1)N_{i} = 4^{k}\prod_{1\leqslant i \leqslant k}r_{i}s_{i}(4r_{i}-1)(4s_{i}-1)\cdot N,$$ with $N=\displaystyle\prod_{1\leqslant i\leqslant k}N_{i}.$
	 Denoting $N = \begin{pmatrix}
	N_{11} & N_{12} & N_{13}\\
	N_{21} & N_{22} & N_{23}\\
	N_{31} & N_{32} & N_{33}\\
	\end{pmatrix}$, it is clear that $N_{11} $ is $O(x^{2k})$, and hence  reduction modulo $p$ gives us that $k$ should be at least $C\log p$ for some $C>0$. 
	Again, we can assume here that $a,b\not\equiv 0 (\mod p)$ and $a,b\not\equiv 1 (\mod p)$: the finitely many excluded primes $p$ are taken into account
	by enlarging constant $C$ if necessary. 
	Thus, the girth of the graph is at least $C\log p$ for some $C>0$.
\end{proof}

\subsection{Diameter} Fix $a\equiv 1 (\mod 3)$ and $b\equiv -1 (\mod  3)$. We shall show that $\langle A^{4}_p,B^{4}_p\rangle=SL_{3}(\mathbb{F}_{p})$ for all sufficiently large primes $p$.
For this, we use a result on the Zariski density and the following proposition.

\begin{proposition}\label{mainprop3}
	Let $a\equiv 1 (\mod 3)$ and $b\equiv -1 (\mod 3)$. Then $\mod 3$ reduction of the matrices $X = A^4 = \begin{pmatrix}
	1 & 4a & 6a^{2} \\
	0 & 1 & 4a \\
	0 & 0 & 1\\
	\end{pmatrix}$ and $Y = B^{4} = \begin{pmatrix}
	1 & 0 & 0 \\
	4b & 1 & 0 \\
	6b^{2} & 4b & 1\\
	\end{pmatrix}$ generate $SL_{3}(\mathbb{F}_{3})$.
\end{proposition}

\begin{proof}
	After reducing $\mod 3,$ we have to show that $$X \equiv  \begin{pmatrix}
		1 & 1 & 0 \\
		0 & 1 & 1 \\
		0 & 0 & 1\\
	\end{pmatrix} (\mod 3) \hbox{ and } Y \equiv \begin{pmatrix}
		1 & 0 & 0 \\
		-1 & 1 & 0 \\
		0 & -1 & 1\\
	\end{pmatrix}(\mod 3)$$ generate $SL_{3}(\mathbb{F}_{3})$. We give an algorithm how to attain elementary matrices of $SL_{3}(\mathbb{F}_{3})$ 
	while taking certain products of our matrices $X^{\pm 1}$ and $Y^{\pm 1}$. This
	goes as follows: 
	\begin{enumerate}
		\item Calculate $C_{1} =  YXY^{-1}X^{-1} \equiv 
		\begin{pmatrix}
		2 & -1 & 1\\
		0 & 1 & 0 \\
		-1 & 0 & 0
		\end{pmatrix}(\mod 3).$
		Similarly,	
			
		$C_{2} = Y^{-1}X^{-1}YX  \equiv \begin{pmatrix}
		2 & 0 & -1\\
		1 & 1 & 0\\
		1 & 0 & 0
		\end{pmatrix} (\mod 3),\quad 
		C_{1}^{-1}=C_{3} \equiv \begin{pmatrix}
		0 & 0 & -1\\
		0 & 1 & 0\\
		1 & 1 & 2
		\end{pmatrix} (\mod 3),$

		$C_{2}^{-1}=C_{4} \equiv \begin{pmatrix}
		0 & 0 & 1\\
		0 & 1 & -1\\
		-1 & 0 & 2
		\end{pmatrix} (\mod 3).$

		\item This implies $C_{1}C_{2}^{-1} = YXY^{-1}X^{-2}Y^{-1}XY \equiv \begin{pmatrix}
		-1 & -1 & 5\\
		0 & 1 & -1 \\
		0 & 0 & -1
		\end{pmatrix} (\mod 3) \Longrightarrow$\\ $C_{1}C_{2}^{-1}X\equiv  \begin{pmatrix}
		-1 & -2 & 4\\
		0 & 1 & 0 \\
		0 & 0 & -1
		\end{pmatrix} (\mod 3)  $ and $(C_{1}C_{2}^{-1}X)^{2} \equiv \begin{pmatrix}
		1 & 0 & 1\\
		0 & 1 & 0 \\
		0 & 0 & 1
		\end{pmatrix} (\mod 3)$.\\ Thus, we can get the matrices $ T_{1}\equiv\begin{pmatrix}
		1 & 0 & 1\\
		0 & 1 & 0 \\
		0 & 0 & 1
		\end{pmatrix} (\mod 3)$ and $ T_{2} \equiv \begin{pmatrix}
		1 & 0 & 0\\
		0 & 1 & 0 \\
		1 & 0 & 1
		\end{pmatrix}(\mod 3).$ 
		
		\item Let $T = [T_{2},T_{1}]\,\footnote{For elements $g,h$ in a group  $G$, the commutator $ [g,h]$ is equal to $g^{-1}h^{-1}gh$. }\equiv  \begin{pmatrix}
		0 & 0 & -1\\
		0 & 1 & 0 \\
		1 & 0 & 0
		\end{pmatrix} (\mod 3)$ and $Z = T\times C_{4} \equiv  \begin{pmatrix}
		1 & 0 & 1\\
		0 & 1 & -1 \\
		0 & 0 & 1
		\end{pmatrix} (\mod 3).$  Then $Z\times T_{1}^{-1}$ gives us $$ \begin{pmatrix}
		1 & 0 & 0\\
		0 & 1 & -1 \\
		0 & 0 & 1
		\end{pmatrix} (\mod 3).$$ 
\end{enumerate}		
Similarly, we get the other elementary matrices. It is standard that the elementary matrices generate $SL_{n}(\mathbb{Z}) $ for all $n\geqslant 1$. Thus, $\mod 3$ reduction of $A^4$ and $B^4$ indeed generate $SL_{3}(\mathbb{F}_{3})$. \end{proof}

\begin{remark}
	Proposition~\ref{mainprop3} sheds light on our choice of the power of $A$ and $B$ to be $4$. The number $3$ in $a\equiv 1 (\mod 3)$ and $b\equiv -1 (\mod 3)$ is not mandatory, any prime $q\geqslant 3$ works. Indeed, if we fix  a prime $q$, $a\equiv 1 (\mod q)$ and $b\equiv -1 (\mod q)$, and take the power of $A$ and $B$ to be $q+1$, then
our algorithm from the preceding proof extends, and we get the elementary matrices in $SL_{3}(\mathbb{F}_{q})$.
\end{remark}

We now state an essential criterion that ensures that $\mod p$ reduction of the matrices $X$ and $Y$ generate $SL_{3}(\mathbb{F}_{p})$ for almost all primes $p$.
We use a formulation from \cite{myfav85} (as we mentioned in Section~\ref{sec:strategy}, this result follows using~\cite{W91} and~\cite{MR735226}). 
We keep the original notation, e.g., here $A$ denotes a subset. 

\begin{proposition}[Lubotzky~\cite{myfav85}]\label{lubprop}
Let $A = \lbrace a_{i}\rbrace _{i\in I}$ be a subset of $SL_{n}(\mathbb{Z})$. Assume that for some prime $p$, the reduction modulo $p$ of $A$ generates the subgroup $SL_{n}(\mathbb{F}_{p})$. If $n=2$ assume $p\neq 2 \text{ or } 3$, if $n = 3 \text{ or } 4,$ assume $p\neq 2$. Then for almost every prime $q$, reduction modulo $q$ of $A$ generates the subgroup $SL_{n}(\mathbb{F}_{q})$.

\end{proposition}

This amazing result gives the existence of a constant $q(n,A,p)$ such that given $n, A,$ and $p$ as in the preceding proposition,
$A$ generates $SL_n(\mathbb{F}_{q})$ for every prime $q\geqslant q(n, A,p)$. The proof in \cite{myfav85} uses the Strong Approximation theorem for linear groups, and
it is not constructive. In particular, it does not provide estimates on the possible value of $q(n, A,p)$. 
However, recent effective variants of the Strong Approximation theorem, see~\cite[Theorem 2.3]{EB2015} and
 \cite[Appendix A]{GoVa2012}, both based on Nori's quantitative proof of the strong approximation~\cite{Nori},
imply that the value of $q(n,A,p)$ is effective, i.e., it can be computed from the given parameters by an algorithm.

\begin{proposition} 
Let $a$ and $b$ be fixed so that  $a\equiv 1 (\mod  3)$  and $b\equiv -1 (\mod 3)$. Then $\langle A^{4}_{p}, B^{4}_{p}\rangle =SL_{3}(\mathbb{F}_{p})$ for almost every prime $p$, i.e., for every prime $p\geqslant q(3,\{X,Y\}, 3)$.
\end{proposition}
\begin{proof}
	Use Proposition \ref{mainprop3} and Proposition \ref{lubprop} to get the existence of $q(3,\{X,Y\}, 3)$.
\end{proof}
Thus, we have that $\lbrace A^{4}, B^{4}\rbrace$ generate a free subgroup in $SL_{3}(\mathbb{Z})$ and also that $\{A_p^4, B_p^4\}$ generate  $SL_{3}(\mathbb{F}_{p})$ for almost all primes $p$. We can now estimate the diameter using a result similar to Proposition~\ref{keyprop2} but for higher dimensions. It was first shown by Helfgott for dimension $3$ and later generalised to all bounded dimensions by Pyber--Szabo~\cite{myfav31} and Breuillard--Green--Tao~\cite{myfav21}. We state it as formulated in \cite{myfav21}.

\begin{proposition}[Breuillard--Green--Tao~\cite{myfav21}, Corollary 2.4]\label{propBrGrTa}
	Let $d \in \mathbb{N}$. Then there are $\epsilon(d) > 0, C_{d} > 0$ such that for
	every absolutely almost simple algebraic group $G$ with $\dim(G)\leqslant d$ defined
	over a finite field $k$, and every finite subset $A$ in $G(k)$ generating $G(k)$, and
	for all $0 < \epsilon < \epsilon(d)$, one of the following two statements holds:
\begin{enumerate}
	\item $|A| >_{d}|G(k)|^{1-C_{d}\epsilon},$ 
	\item $|A^{3}|\geqslant |A|^{1+\epsilon},$\label{eq:BGT2}
\end{enumerate}
	where $X>_{d}Y $ denotes $X> C(d)Y$ and $C(d)$ is some constant depending only on $d$.
\end{proposition}

Since our set $ \lbrace A^{4}, B^{4}\rbrace$ generates a free subgroup in $SL_{3}(\mathbb{Z})$ and the girth of $Cay(SL_{3}(\mathbb{F}_{p}),\lbrace A_{p}^4, B_{p}^4\rbrace)$ is at least $C\log p$ we argue like in the proof of Lemma~\ref{mainlem2} and use at most a constant times Proposition~\ref{propBrGrTa}(\ref{eq:BGT2}) to get a set $S' = S^{O(\log p)}$ with $|S'|\geqslant |G|^{1-\delta}$ for some constant $\delta$. Then applying the following result of Gowers to the subset $S'$ (with the minimal degree of the non-trivial representation as chosen in \cite[Theorem 7.1]{myfav21}), we conclude that $(S')^{3} = SL_{3}(\mathbb{F}_{p})$, and, hence, $$\di Cay(SL_{3}(\mathbb{F}_{p}),\lbrace A^{4}_p, B^{4}_p\rbrace) \hbox{ is } O(\log p),$$
where the implied constant depends on $a$ and $b$.

\begin{proposition}[Gowers \cite{myfav86}, Lemma 5.1;  cf. Nikolov--Pyber~\cite{NiPy2011}, Corollary 1]\label{propGow}
 Let $G$ be a group of order $n$, such that the minimal degree
of a nontrivial representation is $k$. If $A, B, C$ are three subsets of $G$ such
that $|A| |B| |C| >\frac{n^{3}}{k}$, then there is a triple $(a, b, c) \in A \times B \times C$ such that
$ab = c$.
\end{proposition}

Same argument as in the proof of Corollary~\ref{cor:dg2} but using Theorem~\ref{thmgirthsec3} and the preceding conclusion on the diameter gives the required result.
\begin{corollary}\label{cor:dg3}
The sequence $Cay(SL_{3}(\mathbb{F}_{p}),\lbrace A_{p}^4, B_{p}^4\rbrace)$ as $p\to\infty$ is large girth dg-bounded.
\end{corollary}

%%%%%%%%%%%%%%%%%%%%%
\section{Dimension $n\geqslant 4$}\label{sec:4}
%%%%%%%%%%%%%%%%%%%%%

\subsection{Girth}

We first show the following proposition which deals with the case $n=4$.
\begin{proposition}\label{mainpropsec4}
	Let $A = \begin{pmatrix}
	1 & a & 0 & 0\\
	0 & 1 & a & 0\\
	0 & 0 & 1 & a\\
	0 & 0 & 0 & 1\\
	\end{pmatrix}$ and $B = \begin{pmatrix}
	1 & 0 & 0 & 0\\
	b & 1 & 0 & 0\\
	0 & b & 1 & 0\\
	0 & 0 & b & 1
	\end{pmatrix}\in SL_{4}(\mathbb{Z})$ with $a,b\geqslant 2$. Then $\forall l\geqslant 6,\, \langle A^{l},B^{l}\rangle$ is a free subgroup in $SL_{4}(\mathbb{Z})$. 
\end{proposition}
\begin{proof}
	In general, $$A^{k} = \begin{pmatrix}
	1 & {k \choose 1}a & {k \choose 2}a^{2} & {k \choose 3 }a^{3}\\[3pt] 
	0 & 1 & {k \choose 1}a & {k \choose 2}a^{2} \\[3pt] 
	0 & 0 & 1 & {k \choose 1}a \\[3pt] 
	0 & 0 & 0 & 1\\[3pt] 
	\end{pmatrix},$$ where ${k \choose r}$ denotes the usual binomial coefficient.\\
	Fix $l\geqslant 6$. Proceeding as in Proposition \ref{mainpropsec3} we see that, for $r_{i},s_{i}\in \mathbb{Z}\backslash \lbrace 0 \rbrace$, if 
	$$Z_{i} = (A^{l})^{r_{i}}(B^{l})^{s_{i}} = \begin{pmatrix}
	Z_{11_{i}}(a,b) &  Z_{12_{i}}(a,b) & Z_{13_{i}}(a,b) & Z_{14_{i}}(a,b)\\
	Z_{21_{i}}(a,b) &  Z_{22_{i}}(a,b) & Z_{23_{i}}(a,b) & Z_{24_{i}}(a,b)\\
	Z_{31_{i}}(a,b) &  Z_{32_{i}}(a,b) & Z_{33_{i}}(a,b) & Z_{34_{i}}(a,b)\\
	Z_{41_{i}}(a,b) &  Z_{42_{i}}(a,b) & Z_{43_{i}}(a,b) & Z_{44_{i}}(a,b)\\
	\end{pmatrix},$$
	then {\footnotesize $$\begin{pmatrix}
	Z_{11_{i}}(a,b) &  Z_{12_{i}}(a,b) & Z_{13_{i}}(a,b) & Z_{14_{i}}(a,b)\\[3pt] 
	Z_{21_{i}}(a,b) &  Z_{22_{i}}(a,b) & Z_{23_{i}}(a,b) & Z_{24_{i}}(a,b)\\[3pt] 
	Z_{31_{i}}(a,b) &  Z_{32_{i}}(a,b) & Z_{33_{i}}(a,b) & Z_{34_{i}}(a,b)\\[3pt] 
	Z_{41_{i}}(a,b) &  Z_{42_{i}}(a,b) & Z_{43_{i}}(a,b) & Z_{44_{i}}(a,b)\\
	\end{pmatrix} =
	\begin{pmatrix}
		1 & {lr_{i} \choose 1}a & {lr_{i} \choose 2}a^{2} & {lr_{i} \choose 3 }a^{3}\\[3pt] 
	0 & 1 & {lr_{i} \choose 1}a & {lr_{i} \choose 2}a^{2} \\[3pt] 
	0 & 0 & 1 & {lr_{i} \choose 1}a \\[3pt] 
	0 & 0 & 0 & 1\\
	\end{pmatrix}\times \begin{pmatrix}
	1 & 0 & 0 & 0\\[3pt] 
	{ls_{i} \choose 1}b & 1 & 0 & 0 \\[3pt] 
	{ls_{i} \choose 2}b^{2} & {ls_{i} \choose 1}b & 1 & 0 \\[3pt] 
	{ls_{i} \choose 3}b^{3} & {ls_{i} \choose 2}b^{2} & {ls_{i} \choose 1}b & 1\\
	\end{pmatrix} = $$}

	{\footnotesize
	$$ \begin{pmatrix}
	{lr_{i} \choose 3}{ls_{i} \choose 3}a^{3}b^{3} + {lr_{i} \choose 2}{ls_{i} \choose 2}a^{2}b^{2} + {lr_{i} \choose 1}{ls_{i} \choose 1}ab + 1 & {lr_{i} \choose 3}{ls_{i} \choose 2}a^{3}b^{2} + {lr_{i} \choose 2}{ls_{i} \choose 1}a^{2}b + {lr_{i} \choose 1}a & {lr_{i} \choose 3}{ls_{i} \choose 1}a^{3}b + {lr_{i} \choose 2}a^{2} & {lr_{i} \choose 3}a^{3}\\[3pt] 
	{lr_{i} \choose 2}{ls_{i} \choose 3}a^{2}b^{3} + {lr_{i} \choose 1}{ls_{i} \choose 2}ab^{2} + {ls_{i} \choose 1}b & {lr_{i} \choose 2}{ls_{i} \choose 2}a^{2}b^{2} + {lr_{i} \choose 1}{ls_{i} \choose 1}ab + 1 & {lr_{i} \choose 2}{ls_{i} \choose 1}a^{2}b + {lr_{i} \choose 1}a & {lr_{i} \choose 2}a^{2}\\[3pt] 
	{lr_{i} \choose 1}{ls_{i} \choose 3}ab^{3} + {ls_{i} \choose 2}b^{2} & {lr_{i} \choose 1}{ls_{i} \choose 2}ab^{2} + {ls_{i} \choose 1}b & {lr_{i} \choose 1}{ls_{i} \choose 1}ab + 1 & {lr_{i} \choose 1}a\\[3pt] 
	{ls_{i} \choose 3}b^{3} & {ls_{i} \choose 2}b^{2} & {ls_{i} \choose 1}b & 1\\[3pt] 
	\end{pmatrix} $$}

	For $r_{i},s_{i}\in \mathbb{Z}\backslash \lbrace 0 \rbrace, l\geqslant 6,$ we have that the coefficient ${lr_{i} \choose 3}{ls_{i} \choose 3}$ of the highest degree term ($a^{3}b^{3}$) of $Z_{11_{i}}(a,b)$ is non-vanishing. We shall use this fact to show that arbitrary products of $Z_i$'s are non-trivial group elements.	\begin{claim}
		$\forall k\in\mathbb N,\,\prod_{i=1}^kZ_{i} \neq {\rm id}$ in $SL_{4}(\mathbb{Z}).$
	\end{claim}
\begin{proof}[Proof of claim]
We argue by induction on $k$. 
	Clearly in $Z_{1}$, we have $$|Z_{11_{1}}| > |Z_{12_{1}}| + |Z_{13_{1}}| + |Z_{14_{1}}| + 1, $$
	which is the basis of induction. Let the inequality hold for $$Z = Z_{1}Z_{2}\cdots Z_{k} = \begin{pmatrix}
	Z_{11}(a,b) &  Z_{12}(a,b) & Z_{13}(a,b) & Z_{14}(a,b)\\
	Z_{21}(a,b) &  Z_{22}(a,b) & Z_{23}(a,b) & Z_{24}(a,b)\\
	Z_{31}(a,b) &  Z_{32}(a,b) & Z_{33}(a,b) & Z_{34}(a,b)\\
	Z_{41}(a,b) &  Z_{42}(a,b) & Z_{43}(a,b) & Z_{44}(a,b)\\
	\end{pmatrix},$$ 
	i.e., $|Z_{11}| > |Z_{12}| + |Z_{13}| + |Z_{14}| + 1 $. 
	Then for $Z'= Z\times Z_{k+1} = Z\times (A^{l})^{r_{k+1}}(B^{l})^{s_{k+1}},$ we have that
	$$|Z'_{11}| > |Z'_{12}| + |Z'_{12}| + |Z'_{13}| + 1,$$
	by the same argument as in the proof of Claim \ref{mainclaim3} in Proposition \ref{mainpropsec3}.
\end{proof}
	This claim implies that $\forall k\in\mathbb N,\, T = \prod_{i=1}^kA^{lr_{i}}B^{ls_{i}}$ for $l\geqslant 6, r_{i},s_{i}\in \mathbb{Z}\backslash \lbrace 0 \rbrace $ has the property that 
	$$|T_{11}| > |T_{12}| + |T_{13}| + |T_{14}| + 1.$$
	Clearly, this implies that $T \neq B^{s} $ and $T \neq A^{r}$ for any $r,s \in \mathbb{Z}$. Note that $A$ and $B$ are of infinite order in $SL_{4}(\mathbb{Z})$. Thus, we have that $\langle A^{l},B^{l}\rangle $ is a free subgroup of $SL_{4}(\mathbb{Z})$ for any $l\geqslant 6$. 
\end{proof}

For an arbitrary $n\geqslant 4$, we have the following theorem
\begin{theorem}\label{mainthmsec4}
	Let $A = \begin{pmatrix}
	1 & a & 0 & 0& \ldots & 0\\
	0 & 1 & a & 0& \ldots & 0\\
	0 & 0 & 1 & a & \ldots &0\\
	\vdots & & & & & \vdots\\
	&&&&\ldots&a\\
	0 & 0 & 0& 0&\ldots & 1\\
	\end{pmatrix}$ and $B = \begin{pmatrix}
	1 & 0 & 0 & \ldots & &0\\
	b & 1 & 0 & \ldots & &0\\
	0 & b & 1 &  \ldots& &0\\
	0 & 0 & b & \ldots & &0\\
	\vdots &&&&\vdots&\vdots\\
		0& 0 & 0 & \ldots & {\strut \strut b} & 1\\
	\end{pmatrix}\in SL_{n}(\mathbb{Z})$ with $a,b\geqslant 2$. Then $\langle A^{l},B^{l}\rangle \,\, \forall l\geqslant 3(n-1)$ is a free subgroup of $SL_{n}(\mathbb{Z})$.
\end{theorem}
\begin{proof}
		Proceed as in the proof of Proposition \ref{mainpropsec4} or see the Appendix (Section \ref{Appendix}). 
		%in the arXiv version of the present paper\footnote{The Appendix contains both a detailed proof of Proposition \ref{mainpropsec4}
		%and an explicit lower bound on the girth of $\Gamma_p^{n,l}(a,b)=Cay(SL_n(\mathbb F_p), \{ A^l_p, B^l_p\})$ for all $n\geqslant 3$ and $l$ as in I of Main Theorem.}.
\end{proof}

\begin{remark}
	A more refined analysis of the inequalities can show that $\langle A^{l}, B^{l} \rangle,\,\, \forall l> 2n,$ is a free subgroup in $SL_{n}(\mathbb{Z}),$ at the cost of making the proof longer. However, since we are interested in giving explicit examples of dg-bounded graphs of large girth, as long as we give some explicit constant $C(n)>0$ (depending only on dimension $n$) such that $\langle A^{l}, B^{l}\rangle, \,\, \forall l \geqslant C(n),$ is a free subgroup in $SL_{n}(\mathbb{Z})$ we are done. By Theorem~\ref{mainthmsec4}, $C(n)=3(n-1)$ has this property.
\end{remark}
\begin{theorem}\label{thm:girth4}
	Fix $n \geqslant 4, l \geqslant 3(n-1)$. There exists a constant $C>0$ such that for all primes $p$ we have $\gi Cay(\langle A_{p}^{l},B_{p}^{l}\rangle, \{ A_{p}^{l},B_{p}^{l}\})
	 \geqslant C\log p$.
\end{theorem}

\begin{proof}
	By Theorem \ref{mainthmsec4}, $\lbrace A^{l},B^{l}\rbrace $ generate a free subgroup in $SL_{n}(\mathbb{Z})$. It follows that a  product 
	$$(A^{l})^{r_{1}}(B^{l})^{s_{1}}\cdots (A^{l})^{r_{k}}(B^{l})^{s_{k}}$$ can become identity in $SL_{n}(\mathbb{F}_{p})$ only if the entry in the $(1,1)^{\rm th}$ position of $(A^{l})^{r_{1}}(B^{l})^{s_{1}}\cdots (A^{l})^{r_{k}}(B^{l})^{s_{k}}$ is strictly larger than $p$ (again, 
	without loss of generality, 
	$a,b\not\equiv 0 (\mod p)$ and $a,b\not\equiv 1 (\mod p)$).
	Arguing as in the proof of Theorem \ref{thmgirthsec3}, we see that this term has order  $O(a^{nk+1}b^{nk+1})$. It follows that there exists a constant $C$ (independent of $p$) such that $k >C\log p$. Thus, the girth is at least $C\log p$.
\end{proof}

%%%%%%%%%%%%%%%%%%%%
\subsection{Diameter} To estimate the diameter, we first show that there exist infinitely many numbers $l\in \mathbb{N}$ such that $\langle A^{l}_p, B^{l}_p\rangle =SL_{n}(\mathbb{F}_{p})$ for all sufficiently large primes $p$ and fixed $n$.\\

Fix a prime $q$ with $n\equiv 1 (\mod q)$. We have already seen that $\forall l\geqslant 3(n-1), \langle A^{l}, B^{l}\rangle $ is  a free subgroup of $SL_{n}(\mathbb{Z})$, where as usual $a,b\geqslant 2$. We shall first reduce our matrices  $$A = \begin{pmatrix}
	1 & a & 0 & 0& \ldots & 0\\
	0 & 1 & a & 0& \ldots & 0\\
	0 & 0 & 1 & a & \ldots &0\\
	\vdots & & & & & \vdots\\
	&&&&\ldots&a\\
	0 & 0 & 0& 0&\ldots & 1\\
	\end{pmatrix} \text{ and } B = \begin{pmatrix}
	1 & 0 & 0 & \ldots & &0\\
	b & 1 & 0 & \ldots & &0\\
	0 & b & 1 &  \ldots& &0\\
	0 & 0 & b & \ldots & &0\\
	\vdots &&&&\vdots&\vdots\\
		0& 0 & 0 & \ldots & {\strut \strut b} & 1\\
	\end{pmatrix}\in SL_{n}(\mathbb{Z})$$ to  matrices 
	$$A_{q} = \begin{pmatrix}
1 & 1 & 0 & 0&\ldots & 0\\
0 & 1 & 1 & 0& \ldots & 0\\
0 & 0 & 1 & 1 & \ldots & 0\\
\vdots & & & & & \vdots\\
&&&&\ldots&1\\
0 & 0 & 0& 0& \ldots & 1\\
\end{pmatrix} \text{ and } B_{q} = \begin{pmatrix}
1 & 0 & 0 & \ldots & &0\\
	1 & 1 & 0 & \ldots & &0\\
	0 & 1 & 1 &  \ldots& &0\\
	0 & 0 & 1 & \ldots & &0\\
\vdots &&&&\vdots&\vdots\\
0& 0 & 0 & \ldots & {\strut \strut 1} & 1\\
\end{pmatrix}\in SL_{n}(\mathbb{F}_{q}).$$
Obviously, $a,b\equiv 1 (\mod q)$ ensures this. Moreover, keeping this values of $a$ and $b$ fixed, and using a classical result of Lucas, Theorem~\ref{Lucas},
we can guarantee that there are infinitely many powers $l$ of $A,B$ which reduce $\mod q$ to these same matrices in $SL_{n}(\mathbb{F}_{q})$.

\begin{theorem}[Lucas~\cite{MR1505161}]\label{Lucas}
	A binomial coefficient ${\displaystyle \tbinom {\alpha }{\beta}}$ is divisible by a prime $q$ if and only if at least one of the base $q$ digits of $\beta$ is greater than the corresponding digit of $\alpha$.
\end{theorem}

We know that $$A^{k} =  \begin{pmatrix}
1 & {k\choose 1}a & {k \choose 2 }a^{2} & &\ldots & {k\choose n-1}a^{n-1}\\[3pt] 
0 & 1 & {k\choose 1}a & {k \choose 2 }a^{2}& \ldots & {k\choose n-2}a^{n-2}\\[3pt] 
0 & 0 & 1 & {k\choose 1}a & \ldots & {k\choose n-3}a^{n-3}\\[3pt] 
\vdots & & & & & \vdots\\[3pt] 
&&&&\ldots&{k\choose 1}a\\[3pt] 
0 & 0 & 0& 0& \ldots & 1\\
\end{pmatrix}$$
We want the binomial coefficients ${k\choose i} \,\, \forall 1<i\leqslant n-1$ to be divisible by $q$. By Lucas' result, Theorem~\ref{Lucas}, a binomial coefficient $ {\tbinom {k}{i}}$ is divisible by a prime $q$ if and only if at least one of the base $q$ digits of $i$ is greater than the corresponding digit of $k$. So expressing $ k = a_{r}q^{r} +\ldots + a_{1}q + 1\cdot q^{0}$ and the $i$'s for $2\leqslant i \leqslant n-1$ similarly we see that it is sufficient to choose $k$ so that all the $a_{j}$'s are $0$ till the place where the base $q$ representation of $n-1$ ends.
For example, for $q = 2$ and $n=7 = 2^{2}+2+1$ we can take $k$ to be in $\lbrace 1,2^3+1,2^4+1,\ldots\rbrace$. See Theorem \ref{principalthm} for more details.

Thus, we get the matrices  $$A_{q} =\begin{pmatrix}
1 & 1 & 0 & 0&\ldots & 0\\
0 & 1 & 1 & 0& \ldots & 0\\
0 & 0 & 1 & 1 & \ldots & 0\\
\vdots & & & & & \vdots\\
&&&&\ldots&1\\
0 & 0 & 0& 0& \ldots & 1\\
\end{pmatrix} \text{ and } B_{q} = \begin{pmatrix}
1 & 0 & 0 & \ldots & &0\\
	1 & 1 & 0 & \ldots & &0\\
	0 & 1 & 1 &  \ldots& &0\\
	0 & 0 & 1 & \ldots & &0\\
\vdots &&&&\vdots&\vdots\\
0& 0 & 0 & \ldots & {\strut \strut 1} & 1\\
\end{pmatrix}\in SL_{n}(\mathbb{F}_{q})$$ as images under $\mod q$ reduction of matrices $\lbrace A^{k}, B^{k}\rbrace$ which in addition generate a free subgroup in $SL_{n}(\mathbb{Z})$ (if we choose $k$ to be a sufficiently large power bigger than $3(n-1)$). The trick now is to show that these matrices $\{A_q, B_q\}$ generate $SL_{n}(\mathbb{F}_{q})$. For all but $n=4$ this fact follows from the following result (whose proof is non-constructive).

\begin{theorem}[Gow-Tamburini~\cite{GoTa1993}]\label{thm:GoTa}
	If $n\geqslant 2, n\not=4$, then the above matrices $A_{q}$, $B_{q}$ viewed as elements of $SL_{n}(\mathbb{Z})$ generate $SL_{n}(\Z)$. 
	When $n=4$, they generate a subgroup of index $8$ in $SL_{4}(\mathbb{Z})$.
\end{theorem}

For infinitely many values of $n$ in a suitable form (when $n$ is of the form $q^{t}+1, t\in\N$), we give a constructive proof of the required fact. 
Our arguments also handle the exceptional case $n=4$ (because $4=3+1$), where Theorem~\ref{thm:GoTa} does not apply.

\begin{proposition}\label{prop5.8}
	Let $n\geqslant 4$, fix a prime $q$ with $n = q^{t}+1$ and $t\in\N$. Let $$A' = 
	\begin{pmatrix}
1 & 1 & 0 & 0&\ldots & 0\\
0 & 1 & 1 & 0& \ldots & 0\\
0 & 0 & 1 & 1 & \ldots & 0\\
\vdots & & & & & \vdots\\
&&&&\ldots&1\\
0 & 0 & 0& 0& \ldots & 1\\
\end{pmatrix} \text{ and } B'= \begin{pmatrix}
1 & 0 & 0 & \ldots & &0\\
	1 & 1 & 0 & \ldots & &0\\
	0 & 1 & 1 &  \ldots& &0\\
	0 & 0 & 1 & \ldots & &0\\
\vdots &&&&\vdots&\vdots\\
0& 0 & 0 & \ldots & {\strut \strut 1} & 1\\
\end{pmatrix}\in SL_{n}(\mathbb{F}_{q}).$$ Then $\langle A',B'\rangle = SL_{n}(\mathbb{F}_{q})$.
	\begin{comment}
		surject onto $SL_{n}(\mathbb{F}_{p})$ for all sufficiently large primes $p$ where $X_{p},Y_{p}$ are the reductions modulo $p$ of $X$ and $Y$ respectively.
	\end{comment}
\end{proposition}
\begin{proof}
 The prime $q$ is fixed.  We make all the operations in $\mathbb{F}_{q}.$
 Consider the matrices $$A' = \begin{pmatrix}
 1 & 1 & 0 & 0 & 0 & \ldots & 0\\
 0 & 1 & 1 & 0 & 0 & \ldots & 0\\
 0 & 0 & 1 & 1 & 0 & \ldots & 0\\
 \vdots&&&&&&\vdots\\
 0 & \ldots & 0 & 0 & 1 & 1 & 0 \\
 0& \ldots & 0 & 0 & 0 & 1 & 1\\
 0& \ldots & 0 & 0 & 0 & 0 & 1\\
 \end{pmatrix} \text{ and }  B'= \begin{pmatrix}
 1 & 0 & 0 & 0 & 0 &\ldots & 0\\
 1 & 1 & 0 & 0 & 0 & \ldots & 0\\
 0 & 1 & 1 & 0 & 0 & \ldots & 0\\
 \vdots&&&&&&\vdots\\
 0&\ldots & 0 & 1 & 1 & 0 & 0 \\
 0& \ldots. & 0 & 0 & 1 & 1 & 0\\
 0& \ldots & 0 & 0 & 0 & 1 & 1\\
 \end{pmatrix}\in SL_{q^{t}+1}(\mathbb{F}_{q}).$$ Then,  $$(B')^{q^{t}} \equiv \begin{pmatrix}
 1 & 0 & 0 & 0 & 0 & \ldots & 0\\
 0 & 1 & 0 & 0 & 0 & \ldots & 0\\
 0 & 0 & 1 & 0 & 0 & \ldots & 0\\
 \vdots&&&&&&\vdots\\
 0& \ldots & 0 & 0 & 1 & 0 & 0 \\
 0& \ldots & 0 & 0 & 0 & 1 & 0\\
 1& \ldots & 0 & 0 & 0 & 0 & 1\\
 \end{pmatrix} \in \langle A',B'\rangle.$$ We have $$X =B'^{-1}A'(B')^{q^{t}}A'^{-1}B'A' \equiv \begin{pmatrix}
 1 & 1 & 0 & 0 & 0 & \ldots & 0\\
 0 & 1 & 1 & 0 & 0 & \ldots & 0\\
 0 & 0 & 1 & 1 & 0 & \ldots & 0\\
 \vdots&&&&&&\vdots\\
 0 & \ldots & 0 & 0 & 1 & 1 & 0 \\
 0& \ldots & 0 & 0 & 0 & 1 & 0\\
 0& \ldots & 0 & 0 & 0 & 0 & 1\\
 \end{pmatrix}  \in \langle A',B'\rangle $$ 
 
 $$\Rightarrow X_{1} = A'X^{-1} = \begin{pmatrix}
 1 & 0 & 0 & 0 & 0 & \ldots & 0\\
 0 & 1 & 0 & 0 & 0 & \ldots & 0\\
 0 & 0 & 1 & 0 & 0 & \ldots & 0\\
 \vdots&&&&&&\vdots\\
 0 & \ldots & 0 & 0 & 1 & 0 & 0 \\
 0& \ldots & 0 & 0 & 0 & 1 & 1\\
 0& \ldots & 0 & 0 & 0 & 0 & 1\\
 \end{pmatrix} \in \langle A',B'\rangle.$$ Similarly, $$Y_{1} = \begin{pmatrix}
 1 & 0 & 0 & 0 & 0 & \ldots. & 0\\
 0 & 1 & 0 & 0 & 0 &\ldots & 0\\
 0 & 0 & 1 & 0 & 0 & \ldots & 0\\
 \vdots&&&&&&\vdots\\
 0 & \ldots & 0 & 0 & 1 & 0 & 0 \\
 0& \ldots & 0 & 0 & 0 & 1 & 0\\
 0& \ldots & 0 & 0 & 0 & 1 & 1\\
 \end{pmatrix} \in \langle A',B'\rangle .$$ The matrices $\begin{pmatrix}
 1 & 1\\
 0 & 1 
 \end{pmatrix}, \begin{pmatrix}
 1 & 0\\
 1 & 1\\
 \end{pmatrix}$ generate $SL_{2}(\mathbb{Z})$, hence we have 
 
 $$ Y = \begin{pmatrix}
 1 & 0 & 0 & 0 & 0 & \ldots & 0\\
 0 & 1 & 0 & 0 & 0 & \ldots & 0\\
 0 & 0 & 1 & 0 & 0 & \ldots & 0\\
 \vdots&&&&&&\vdots\\
 0 &\ldots& 0 & 0 & 1 & 0 & 0 \\
 0& \ldots & 0 & 0 & 0 & x & y\\
 0&\ldots & 0 & 0 & 0 & z & t\\
 \end{pmatrix} \in \langle A',B'\rangle, \forall x,y,z,t \text{ with } xt-yz \equiv 1 (\mod q).$$ This implies 
 
 $$A'\times Y =  {\tiny \begin{pmatrix}
 1 & 1 & 0 & 0 & 0 & \ldots & 0\\
 0 & 1 & 1 & 0 & 0 & \ldots & 0\\
 0 & 0 & 1 & 1 & 0 & \ldots & 0\\
 \vdots&&&&&&\vdots\\
 0 & \ldots & 0 & 0 & 1 & 1 & 0 \\
 0& \ldots & 0 & 0 & 0 & 1 & 1\\
 0& \ldots& 0 & 0 & 0 & 0 & 1\\
 \end{pmatrix} \times \begin{pmatrix}
 1 & 0 & 0 & 0 & 0 & \ldots & 0\\
 0 & 1 & 0 & 0 & 0 & \ldots & 0\\
 0 & 0 & 1 & 0 & 0 & \ldots & 0\\
 \vdots&&&&&&\vdots\\
  0 & \ldots & 0 & 0 & 1 & 0 & 0 \\
 0& \ldots & 0 & 0 & 0 & x & y\\
 0& \ldots & 0 & 0 & 0 & z & t\\
 \end{pmatrix} = \begin{pmatrix}
 1 & 1 & 0 & 0 & 0 & \ldots & 0\\
 0 & 1 & 1 & 0 & 0 & \ldots & 0\\
 0 & 0 & 1 & 1 & 0 & \ldots & 0\\
 \vdots&&&&&&\vdots\\
 0 & \ldots & 0 & 0 & 1 & x & y \\
 0& \ldots & 0 & 0 & 0 & x+z & t+y\\
 0& \ldots & 0 & 0 & 0 & z & t\\
 \end{pmatrix}} \in \langle A',B'\rangle.$$ Choosing $x = 0, y=-1, t=1,z =1$ we have $xt-yz \equiv 1 (\mod q)$ we get $$Z = {\tiny \begin{pmatrix}
 1 & 1 & 0 & 0 & 0 & \ldots & 0\\
 0 & 1 & 1 & 0 & 0 & \ldots & 0\\
 0 & 0 & 1 & 1 & 0 & \ldots & 0\\
\vdots&&&&&&\vdots\\
 0 & \ldots & 0 & 0 & 1 & 0 & -1 \\
 0& \ldots & 0 & 0 & 0 & 1 & 0\\
 0& \ldots & 0 & 0 & 0 & 1 & 1\\
 \end{pmatrix} \in \langle A',B'\rangle}, \hbox{ and hence } Y_{1}^{-1}Z = {\tiny\begin{pmatrix}
 1 & 1 & 0 & 0 & 0 & \ldots & 0\\
 0 & 1 & 1 & 0 & 0 & \ldots & 0\\
 0 & 0 & 1 & 1 & 0 & \ldots & 0\\
 \vdots&&&&&&\vdots\\
 0 & \ldots & 0 & 0 & 1 & 0 & -1 \\
 0& \ldots & 0 & 0 & 0 & 1 & 0\\
 0& \ldots & 0 & 0 & 0 & 0 & 1\\
 \end{pmatrix}  \in \langle A',B'\rangle.}$$ 
 
 Also, $$[X^{-1},X_{1}^{-1}]={\tiny \begin{pmatrix}
 1 & 0 & 0 & 0 & 0 & \ldots & 0\\
 0 & 1 & 0 & 0 & 0 &\ldots & 0\\
 0 & 0 & 1 & 0 & 0 & \ldots & 0\\
 \vdots&&&&&&\vdots\\
 0 & \ldots & 0 & 0 & 1 & 0 & 1\\
 0& \ldots & 0 & 0 & 0 & 1 & 0\\
 0& \ldots & 0 & 0 & 0 & 0 & 1\\
 \end{pmatrix}} \in \langle A',B'\rangle.$$
Let $ T = [X^{-1},X_{1}^{-1}]Y_{1}^{-1}Z = {\tiny\begin{pmatrix}
 1 & 1 & 0 & 0 & 0 & \ldots & 0\\
 0 & 1 & 1 & 0 & 0 & \ldots & 0\\
 0 & 0 & 1 & 1 & 0 & \ldots & 0\\
 \vdots&&&&&&\vdots\\
 0 & \ldots & 0 & 0 & 1 & 0 & 0 \\
 0& \ldots & 0 & 0 & 0 & 1 & 0\\
 0& \ldots & 0 & 0 & 0 & 0 & 1\\
 \end{pmatrix}} \in \langle A,B\rangle$. Then,
 $$ A'T^{-1} = {\tiny \begin{pmatrix}
 1 & 0 & 0 & 0 & 0 & \ldots & 0\\
 0 & 1 & 0 & 0 & 0 & \ldots & 0\\
 0 & 0 & 1 & 0 & 0 & \ldots & 0\\
 \vdots&&&&&&\vdots\\
 0 & \ldots & 0 & 0 & 1 & 1 & 0 \\
 0& \ldots & 0 & 0 & 0 & 1 & 1\\
 0& \ldots & 0 & 0 & 0 & 0 & 1\\
 \end{pmatrix}} \in \langle A',B'\rangle \Longrightarrow X_{1}^{-1}AT^{-1} = {\tiny\begin{pmatrix}
 1 & 0 & 0 & 0 & 0 & \ldots & 0\\
 0 & 1 & 0 & 0 & 0 & \ldots & 0\\
 0 & 0 & 1 & 0 & 0 & \ldots & 0\\
 \vdots&&&&&&\vdots\\
 0 & \ldots & 0 & 0 & 1 & 1 & 0 \\
 0& \ldots & 0 & 0 & 0 & 1 & 0\\
 0& \ldots & 0 & 0 & 0 & 0 & 1\\
 \end{pmatrix}} \in \langle A',B'\rangle .$$ In such a way, we get the standard elementary matrix $E_{i,i+1}\in \langle A',B'\rangle\, \forall 1\leqslant i\leqslant n-1$. Now, it is easy to see that we can get all the elementary matrices $E_{i,j}, 1\leqslant i,j\leqslant n$ and since elementary matrices generate $SL_{n}(\mathbb{F}_{q})$, we are done. 
\end{proof}

\begin{proposition}\label{prop5.9}
	Let $n\geqslant 2$ be any integer and let $q$ be a prime with $n\equiv 1 (\mod q)$. Let $$A' = 
	\begin{pmatrix}
	1 & 1 & 0 & 0&\ldots & 0\\
	0 & 1 & 1 & 0& \ldots & 0\\
	0 & 0 & 1 & 1 & \ldots & 0\\
	\vdots & & & & & \vdots\\
	&&&&\ldots&1\\
	0 & 0 & 0& 0& \ldots & 1\\
	\end{pmatrix} \text{ and } B'= \begin{pmatrix}
	1 & 0 & 0 & \ldots & &0\\
	1 & 1 & 0 & \ldots & &0\\
	0 & 1 & 1 &  \ldots& &0\\
	0 & 0 & 1 & \ldots & &0\\
	\vdots &&&&\vdots&\vdots\\
	0& 0 & 0 & \ldots & {\strut \strut 1} & 1\\
	\end{pmatrix}\in SL_{n}(\mathbb{F}_{q}).$$ Then, $\langle A',B'\rangle = SL_{n}(\mathbb{F}_{q})$.
\end{proposition}

\begin{proof}
	We apply Theorem \ref{thm:GoTa} to $A',B'$ over $\Z$ instead of $\mathbb{F}_{q}$. Then we do the
	$\mod q$ reduction $SL_{n}(\mathbb Z)\twoheadrightarrow SL_{n}(\mathbb{F}_{q})$ and conclude by Proposition \ref{prop5.8}.
\end{proof}

\begin{theorem}\label{principalthm}
	For each integer $n\geqslant 4$, let $q$ be a prime with $n\equiv 1 (\mod q)$ and let	 $$A = \begin{pmatrix}
	1 & a & 0 & 0&\ldots & 0\\
	0 & 1 & a & 0& \ldots & 0\\
	0 & 0 & 1 & a & \ldots &0\\
	\vdots&&&&&\vdots\\
	\\
	0 & 0 & 0 & 0& \ldots  & 1\\
	\end{pmatrix}\text{ and } B = \begin{pmatrix}
	1 & 0 & 0 & \ldots & 0\\
	b & 1 & 0 & \ldots & 0\\
	0 & b & 1 & \ldots & 0\\
	\vdots&&&&\vdots\\
	\\
	0 & 0 & \ldots & b & 1\\
	\end{pmatrix}\in SL_{n}(\mathbb{Z})$$ with $a,b\geqslant 2$ and $a,b \equiv 1(\mod q)$. Then there exist constants $K = K(n,a,b), K_{1}=K_1(n,a,b)$ such that for all primes $p> K$, the $\mod p$ reduction  of $S = \lbrace (A^{q^{k+2}+1})^{\pm 1}, (B^{q^{k+2}+1})^{\pm 1}\rbrace $ with $k\in \mathbb{N}, \, k\geqslant t$ 
	 with $t\in\N$ given by  $q^t\leqslant n <q^{t+1}$, 
	generate $SL_{n}(\mathbb{F}_{p})$ and  the diameter-by-girth ratio of the sequence of Cayley graphs $Cay(SL_{n}(\mathbb{F}_{p}),S)$  is less than $K_{1}$.  
\end{theorem}

\begin{proof}
	The proof has two parts.
	
	\begin{enumerate}
		\item \underline{Generation}: For each  integer $n\geqslant 4$,  we are given a prime $q$ with $n\equiv 1 (\mod q)$, and
		the matrices $A$ and $B$ in $SL_{n}(\mathbb{Z})$. Let $t$ denote the highest power of $q$ in the base $q$ representation of $n$. We would like to obtain the matrices  $$A_{q} =\begin{pmatrix}
		1 & 1 & 0 & 0&\ldots & 0\\
		0 & 1 & 1 & 0& \ldots & 0\\
		0 & 0 & 1 & 1 & \ldots & 0\\
		\vdots & & & & & \vdots\\
		&&&&\ldots&1\\
		0 & 0 & 0& 0& \ldots & 1\\
		\end{pmatrix}\text{ and }B_{q} = \begin{pmatrix}
		1 & 0 & 0 & \ldots & &0\\
		1 & 1 & 0 & \ldots & &0\\
		0 & 1 & 1 &  \ldots& &0\\
		0 & 0 & 1 & \ldots & &0\\
		\vdots &&&&\vdots&\vdots\\
		0& 0 & 0 & \ldots & {\strut \strut 1} & 1\\
		\end{pmatrix}\in SL_{n}(\mathbb{F}_{q})$$
		as words in $A$ and $B$ reduced modulo $q$. By Lucas result, Theorem~\ref{Lucas}, we see that $$A^{rq^{(t+1)}+1} (\mod q)  = A_{q} \text{ and } B^{rq^{(t+1)}+1}(\mod q) = B_{q} \text{ in } SL_{n}(\mathbb{F}_{q})$$
		for all integers $r$.
		By Proposition \ref{prop5.9}, it follows that the group generated by these matrices coincides with $SL_{n}(\mathbb{F}_{q})$. By Proposition \ref{lubprop}, we have $\langle A^{rq^{(t+1)}+1},B^{rq^{(t+1)}+1} \rangle =SL_{n}(\mathbb{F}_{p})$ for all positive integers $r$ and all primes $p > K$, where $K$ is a constant.\medskip
		\item \underline{Freeness}: By Theorem \ref{mainthmsec4}, we know that $\langle A^{l},B^{l}\rangle$ is free in $SL_{n}(\mathbb{Z})$ for all $l\geqslant 3(n-1)$. This implies that the girth of the corresponding sequence of Cayley graphs  is at least $C_{2}\log p$ for some constant $C_{2}>0$.
	\end{enumerate}
     It remains to make sure that $rq^{t+1}+1\geqslant 3(n-1)$. If $q\geqslant 3$, then choosing $r$ to be equal to positive powers of $q$ gives us the required bound. For $q=2$, choose $r=q^{2}$.
 
 Since $S$ generates a free subgroup in $SL_{n}(\mathbb{Z})$, by Proposition \ref{propBrGrTa}, the diameter of $SL_{n}(\mathbb{F}_{p})$ with respect to the $\mod p$ reduction of generators from $S$ is $\leqslant C_{1}\log p,$ where $C_{1}>0$ is a constant. We know already that the girth is at least $C_{2}\log p$. Thus, the diameter-by-girth ratio is $\leqslant\frac{C_{1}}{C_{2}} = K_{1}$.\end{proof}

\begin{lemma}[Effectiveness of the constants]
	The constants $K, L, c_n$ and that of $O(\log p)$ term of Main Theorem are effective.
\end{lemma}
\begin{proof}
	The constants $K$ and $L$ are effective by recent works of Breuillard \cite[Theorem 2.3]{EB2015} and Golsefidy-Varj\'{u} \cite[Appendix A]{GoVa2012}, see our explanation following Proposition \ref{lubprop}. The constant $c_{n}$ is also effective, \cite[section 6]{M82}. For the constant in the $O(\log p)$ term of the upper bound on the diameter (let us call it $d_{n}$), we proceed as follows. Given a generating set $S = \{A^{\pm l}_{p}, B^{\pm l}_{p}\}$ of $G = SL_{n}(\mathbb{F}_{p})$, using our results on
	the large girth of $Cay(SL_{n}(\mathbb{F}_{p}), \lbrace A^l_{p}, B^l_{p}\rbrace)$, we have $|S^{\frac{c_{n}}{6}\log p}|\geqslant 3^{\frac{c_{n}}{6}\log p} = |G|^{t}$ for some $ 0 < t < 1$. Taking $A = S^{\frac{c_{n}}{6}\log p}$ and using a result of Pyber-Szab{\'o} \cite[Theorem 2]{myfav31}, applied to $SL_{n}(\mathbb{F}_{p})$, we have 
	$$\hbox{ either } |A^3| > |A|^{1+\epsilon(n)}, \text{ or } A^{3} = G, $$
	where $\epsilon(n)$ is effective.
	 If we are in the former case then applying the above inequality $k$ times we will fall into the latter case whenever $k$ is large enough.
	It follows from $|A| \geqslant |G|^t$ and Proposition~\ref{propGow} that 
	 such a constant $k$ is given by $t(1+\epsilon(n))^{k} = 99/100 $. Since $\epsilon(n)$ is effective, we conclude that $k$, and hence $d_{n} = 3k$, are effective.	
\end{proof}

Combining results proved throughout Sections~\ref{sec:2}--\ref{sec:4}, we obtain all the statements of our Main Theorem except expansion.
As explained in Section~\ref{sec:strategy}, the fact that our graphs $\Gamma_p^{n,l}(a,b)$ as $p\to\infty$ are indeed expanders is a by-product of our results about freeness (for $n=2$) and
about freeness and generation $\mod p$ (for $n\geqslant 3$), using~\cite{MR2415383} and \cite{MR2897695}, respectively.

%%%%%%%%%%%%%%%%%%%%%%%%%%%%%%%%%%%%%%%%%%%%%%%%%%
\section{Further results and questions}\label{sec:questions}

In dimension $n\geqslant 3$, every free subgroup $\langle A^l,B^l\rangle$ from our Main Theorem is an explicit example of 
a \emph{thin} matrix group which is, by definition, a finitely generated subgroup of $GL_n(\mathbb Z)$ which is of infinite index in the $\mathbb Z$-points of its Zariski closure in $GL_n$. 
Indeed,  it follows from our proof that the integral Zariski closure of $\langle A^l,B^l\rangle$
is $SL_n(\mathbb Z)$, see Section~\ref{sec:strategy}. On the other hand, $\langle A^l,B^l\rangle$ is of infinite index in $SL_n(\mathbb Z)$ because, for $n\geqslant 3$, $SL_n(\mathbb Z)$ 
is not virtually free. For applications of thin matrix groups in many diophantine and geometric problems, see~\cite{Sar} and references therein. For a recent characterization of thin matrix groups, see~\cite{LV} and for 
a concise invitation and examples in dimension 2, 3 and 4, see~\cite{KLLR}.  Our Main Theorem gives infinitely many explicit examples of thin matrix groups in each dimension $n\geqslant 3.$

\begin{corollary}[$2k$-regular logarithmic girth expanders]\label{cor:lub}
Let $n\geqslant 2$ and $l\geqslant 1$ be as in III of Main Theorem. For every integer $k\geqslant 2$, there exist $S_1,\ldots, S_{k}\in \langle A^l, B^l\rangle$ such that the sequence of Cayley graphs
$\Gamma_p^{n,l}(a,b;k)=Cay(SL_n(\mathbb F_p), \{(S_1)_p,\ldots, (S_{k})_p\})$ as $p\to\infty$ 
is a $2k$-regular large girth dg-bounded graph. Moreover, it is a $2k$-regular logarithmic girth expander.
\end{corollary}
\begin{proof}
For $k=2$, we set $S_1=A^l, S_2=B^l$ since
4-regular graphs $\Gamma_p^{n,l}(a,b)=Cay(SL_n(\mathbb F_p), \{ A^l_p, B^l_p\})$ as $p\to\infty$
from our Main theorem form such expander. For each  $k\geqslant 3$, take a subgroup of index $k-1$ in $\langle A^l, B^l\rangle\leqslant SL_n(\mathbb Z)$.
Since $\langle A^l, B^l\rangle$ is free and Zariski dense, this subgroup is also free and Zariski dense. By Nielsen-Schreier formula~\cite[Ch.I, Prop.3.9]{myfav84}, it has rank 
$k$. We denote the subgroup by $F_k$ and its free generators by $S_1, \ldots, S_k$. The Cayley graph of $F_k$ with respect to these free generators is $2k$-regular. 
By the Matthews-Vaserstein-Weisfeiler theorem~\cite{MR735226}, for all sufficiently large prime numbers $p$, the $\mod p$ reductions of $S_1, \ldots, S_k$ generate the entire $SL_n(\mathbb F_p)$  as $F_k$ is Zariski dense.
Therefore, by same arguments as in the previous sections, we conclude that the diameter of $Cay(SL_n(\mathbb F_p), \{(S_1)_p,\ldots, (S_{k})_p\})$ is $O(\log p)$
as $F_k$ is free and the entire $SL_n(\mathbb F_p)$ is generated,
its girth is logarithmic as $F_k$ is free, and the sequence  is an
expander as $p\to\infty$ because $F_k$ is Zariski dense.  \end{proof}

Up to a slight modification both in the formulation and in the proof of the preceding corollary, 
we obtain a large girth dg-bounded expander sequence of 2k-regular congruence quotients. 
Let $\ell\in \mathbb Z$ be an arbitrary integer and $SL_n(\mathbb Z)\twoheadrightarrow SL_n(\mathbb Z / \ell \mathbb Z): S_i\mapsto (S_i)_\ell$
be the congruence surjection.

\begin{corollary}[$2k$-regular logarithmic girth congruence quotients]\label{cor:lub2}
Let $n\geqslant 2$ and $l\geqslant 1$ be as in III of Main Theorem. For every integer $k\geqslant 2$, there exist $S_1,\ldots, S_{k}\in \langle A^l, B^l\rangle$ such that the sequence of Cayley graphs
$\Lambda_\ell^{n,l}(a,b;k)=Cay(\langle (S_1)_\ell, \ldots, (S_k)_\ell\rangle, \{(S_1)_\ell,\ldots, (S_{k})_\ell\})$ as $\ell\to\infty$ 
is a $2k$-regular large girth dg-bounded graph. Moreover, it is a $2k$-regular logarithmic girth expander and
there exists an integer $\ell_0$ with $\langle (S_1)_\ell, \ldots, (S_k)_\ell\rangle=SL_n(\mathbb Z / \ell \mathbb Z)$  if $\ell$ is coprime with $\ell_0$. 
\end{corollary}
\begin{proof}
Let $S_1, \ldots, S_k$ be as in the proof of Corollary~\ref{cor:lub} so that $\langle S_1, \ldots, S_k\rangle\leqslant SL_n(\mathbb Z)$ is free of rank $k$ and Zariski dense.
The logarithmic girth follows as previously from freeness, the expansion and the existence of $\ell_0$ is by~\cite[Theorem 1]{MR2897695}.
For $\ell$ coprime with $\ell_0$, the logarithmic upper estimates on the diameter, and hence, dg-boundedness of $\Lambda_\ell^{n,l}(a,b;k)$ follows from the expansion.
For an arbitrary $\ell$, when possibly $\langle (S_1)_\ell, \ldots, (S_k)_\ell\rangle< SL_n(\mathbb Z / \ell \mathbb Z)$ is a proper subgroup, such estimates are immediate from expansion.
\end{proof}

The matrices $S_1,\ldots, S_k$ in Corollaries \ref{cor:lub} and \ref{cor:lub2} can be given explicitly as concrete words in generators $A^l$ and $B^l$ (it is easy to produce generators of a finite index subgroup in a free group).
However, in Corollary~\ref{cor:lub2}, algebraic features of a proper subgroup $\langle (S_1)_\ell, \ldots, (S_k)_\ell\rangle< SL_n(\mathbb Z / \ell \mathbb Z)$ 
can vary and the subgroup itself is not so explicit, whence our use of expansion for diameter estimates in this case. However, we believe that a combination of our strategy
with analysis of such Zariski dense subgroups from~\cite[Theorem 1]{MR2897695} can yield the required estimates on the diameter with no use of expansion properties of the involved graphs.

Our graphs are in all dimensions $n\geqslant 2$ and clearly not isomorphic to each other whenever the dimensions are different.
Moreover, we can make them distinct from the large-scale geometry point of view.
Indeed, taking suitable subsequences in $\Gamma_p^{n,l}(a,b)=Cay(SL_n(\mathbb F_p), \{ A^l_p, B^l_p\})$ as $p\to\infty$
yields graphs in distinct regular\footnote{A map between graphs is \emph{regular} if it is Lipschitz and pre-images of vertices have uniformly bounded cardinality. Two graphs are \emph{regularily equivalent} if 
there exist two regular maps: from one graph to the other, and back.} equivalence classes; subsequences in a given dimension $n$ or in distinct dimensions $n_1, \ldots, n_N$.
Therefore, our large girth dg-bounded Cayley graphs of $SL_n(\mathbb F_p)$ as $p\to\infty$ viewed for each $n\geqslant 2$ over a suitable subsequence of primes are not coarsely equivalent to each other. These are the first such explicit examples in all dimensions.

\begin{corollary}[regularly/coarsely distinct logarithmic girth expanders] Let $n,n_1, \ldots, n_N\geqslant 2$ arbitrary dimensions and $l\geqslant 1$ as in III of Main Theorem, $N\in\mathbb N$. Let $\mathbb P$ be the set of all primes.
\begin{itemize}
\item[(i)] There exists an infinite subset $P\subseteq\mathbb P$ such that 
for any infinite subsets $Q, Q'\subseteq P$ with $Q\setminus Q'$ infinite, there is no regular map from
$\Gamma_p^{n,l}(a,b)$ as $p\to\infty$, $p\in Q$ to $\Gamma_p^{n,l}(a,b)$ as $p\to\infty$, $p\in Q'$.
\item[(ii)] There exist infinite subsets $P_1,\ldots, P_N\subseteq \mathbb P$ such that for any infinite subsets $Q, Q'\subseteq \cup_{i=1}^NP_i$ with $Q\setminus Q'$ infinite, there is no regular map from 
$\Gamma_p^{n_i,l}(a,b)$ as $p\to\infty$, $p\in Q$ to  $\Gamma_p^{n_j,l}(a,b)$ as $p\to\infty$, $p\in Q'$.
\end{itemize}

\end{corollary}
\begin{proof}
(i) We choose $P\subseteq\mathbb P$ such that for each $p\in P$ and  the next prime in the subsequence $q\in P$,  we have $\gi \Gamma_q^{n,l}(a,b)>|\Gamma_p^{n,l}(a,b)|$ and $|\Gamma_q^{n,l}(a,b)| / |\Gamma_p^{n,l}(a,b)|\to\infty$ as $p\to\infty$.
This is possible as the graph is large girth. Then, the statement is immediate by Theorem 2.8 of \cite{Hume} applied to $\Gamma_p^{n,l}(a,b)$ as $p\to\infty, p\in P$.
This yields $2^{\aleph_0}$ regular equivalence classes of large girth dg-bounded expanders in each dimension $n$.

(ii) We
take the union $\Gamma_p=\cup_{i=1}^N \Gamma_p^{n_i,l}(a,b)$ as $p\to\infty$, ${p\in \cup_{i=1}^N\mathbb P}$ of $N$ sequences of our graphs 
in the chosen dimensions $n_1, \ldots, n_N$.  Since the graphs are large girth
and the assumptions  on $\gi$ and cardinality required by~\cite[Theorem 2.8]{Hume} are transitive,
we can choose the required infinite subsets  $P_1,\ldots, P_N\subseteq \mathbb P$ successively  
for $n_1, \ldots, n_N$.
\end{proof}
There is much flexibility in the formulation of the preceding corollary. In particular, there are numerous 
choices for subsets $P, P_1, \ldots, P_N$ and 
parameters $l, a,b$ can vary for distinct dimensions. The analogous result holds for graphs
$\Gamma_p^{n,l}(a,b;k)$ and $\Lambda_\ell^{n,l}(a,b;k)$ defined in Corollaries \ref{cor:lub} and \ref{cor:lub2}.
 
%%%%%%%%%%

%%%%%%%%%
The following question is highly intriguing. Again,
$n\geqslant 2$ and $l\geqslant 1$ are as in III of Main Theorem. An expander is called a \emph{super-expander} if it is an expander
with respect to every super-reflexive Banach space~\cite{MN14}. In particular, such a graph does not coarsely embed into any uniformly convex Banach space. 
\begin{question}[super-expansion]\label{q:sexp}
Is $\Gamma_p^{n,l}(a,b)=Cay(SL_n(\mathbb F_p), \{ A^l_p, B^l_p\})$ as $p\to\infty$ a super-expander?
\end{question}

This is open for $n=2$ and $l=1$, hence, also for Margulis' expander~\cite{M82}. 
Currently available super-expanders, produced using a strong Banach variant of Kazhdan's property~(T)~\cite{Laf08,Laf09},
an iterative zig-zag type combinatorial construction~\cite{MN14}, or by means of warped cones, see e.g.,~\cite{NS17}, are all of finite girth.
A positive answer to Question~\ref{q:sexp} for at least one choice of parameters $n$ and $l$ will allow, for instance, to 
build an infinite `super monster' \emph{group} (like that from~\cite{GRW, ArzDel08} but with respect to group actions on super-reflexive Banach spaces):
a finitely generated group which does not admit a coarse embedding into any uniformly convex Banach space. This is of great interest in the context of the Novikov conjecture~\cite{KYu06}.

Question~\ref{q:sexp} is a large girth counterpart of well-known question, for each $n\geqslant 3$, whether or not the sequence of congruence quotients 
$Cay(SL_n(\mathbb Z/\ell\mathbb Z), S_\ell)$, as $\ell\to\infty$ is a super-expander; $S_\ell$ denotes the canonical image of a finite generating set $S$ of $SL_n(\mathbb Z)$.
However, for each $n\geqslant 3$, our expander
$Cay(SL_n(\mathbb F_p), \{ A^l_p, B^l_p\})$ as $p\to\infty$ is not coarsely equivalent to $Cay(SL_n(\mathbb F_p), S_p)$ as $p\to\infty$. 
Indeed, the sequence of marked finite groups $(SL_n(\mathbb F_p), S_p)$ as $p\to\infty$ converges to $(SL_n(\mathbb Z), S)$,
then by~\cite[Corollary 5]{Kun}, the sequence $Cay(SL_n(\mathbb F_p), S_p)$ as $p\to\infty$ is not coarsely embeddable into our sequence
$Cay(SL_n(\mathbb F_p), \{ A^l_p, B^l_p\})$ as $p\to\infty$. Therefore, 
a conjectural Banach property (T) of $SL_n(\mathbb Z), n\geqslant 3$ does not apply to conclude super-expansion of our expander (although, it would apply to the congruence quotients)  and
entirely new methods have to be designed in order to answer Question~\ref{q:sexp}. 
\smallskip

Since all three explicit constructions of 
large girth dg-bounded Cayley graphs, Margulis'\cite{M82}, Lubot\-zky-Phillips-Sarnak's~\cite{LPS}, and ours,
happen to be expanders and, in addition, of logarithmic girth, in the next question we wonder if this is always the case.
Restricting to Cayley graphs of finite quotients of a non Zariski dense subgroup of $SL_n(\mathbb Z)$ or of $Sp_{2n}(\mathbb Z)$ (cf.~\cite{Sp}),
or of another algebraic group are interesting instances of this question.
\begin{question}[large girth dg-bounded graphs with no expansion]\strut
Does there exist a large girth dg-bounded graph made of $r\geqslant 3$ regular Cayley graphs that is not an expander?
Moreover, with no weakly embedded expander? Furthermore, that is not a generalized expander?
\end{question}
Our final question explores possible metric embeddings differences between random graph expanders 
and known explicit constructions of large girth dg-bounded expanders. We refer to~\cite{MendelNaor} for the terminology and for the amazing results which 
yielded the question to us.
\begin{question}[random vs explicit]
Does there exist a Hadamard space $(M,d_M)$ such that\linebreak $\Gamma_p^{n,l}(a,b)=Cay(SL_n(\mathbb F_p), \{ A^l_p, B^l_p\})$ as $p\to\infty$ is an expander with respect to $(M,d_M)$ yet
a random regular graph is not expander with respect to $(M,d_M)$?
\end{question}
The main outcome of~\cite{MendelNaor} is a Hadamard space $(N, d_N)$  and a sequence  of 3-regular graphs $(\Lambda_n)_{n\in\mathbb N}$
that is an expander with respect to $(N,d_N)$ yet a random regular graph is not an expander with respect to $(N,d_N)$.
The construction of graphs $(\Lambda_n)_{n\in\mathbb N}$ is by a zig-zag iteration and it
is neither large girth nor made of Cayley graphs. The Hadamard space $(N,d_N)$ is the Euclidean cone over 
a suitable large girth dg-bounded graph (obtained from a random regular graph by removing a portion of edges). If our graph $\Gamma_p^{n,l}(a,b)$ as $p\to\infty$ is an expander with respect to  this $(N, d_N)$, then we have an affirmative answer to the preceding question.
This would give the first  large girth example of this kind versus the `small girth' construction from~\cite{MendelNaor}. In addition, a positive answer to the preceding question
(with a Hadamard space $(M,d_M)$ that differs from a Hilbert space and that is possibly not  such a cone) would also 
allow us to apply the main result of~\cite{NS}  to our graphs and such a space $(M,d_M)$. 
This would yield first examples of groups with strong fixed point properties on such $(M,d_M)$: namely, finitely generated groups such that,
almost surely, any of its isometric action on $(M,d)$  has a common fixed point.

\section{Appendix}\label{Appendix}
For an interested reader, we give a detailed proof of  Theorem \ref{mainthmsec4}.

\begin{proof}[Proof of Theorem \ref{mainthmsec4}]
	Fix $(n-1) = k\in \mathbb{N}$ and consider the matrix $A^{lr_{i}}B^{ls_{i}}$ with $l \geqslant 3k$ and $r_{i},s_{i}\in \mathbb{Z}\backslash \lbrace 0\rbrace $. Let $a,b\in \mathbb{N}$ with $a,b\geqslant 2$. Suppose 
	$$\mathcal{P}_{i} = A^{lr_{i}}B^{ls_{i}} = \begin{pmatrix}
	P_{11}(a,b) & P_{12}(a,b) & \ldots & P_{1n}(a,b)\\
	P_{21}(a,b) & P_{22}(a,b)  & \ldots & P_{2n}(a,b)\\
	P_{31}(a,b) & P_{32}(a,b)  & \ldots & P_{3n}(a,b)\\
	\vdots & & & \vdots\\
	&&\ldots&P_{(n-1)n}(a,b)\\
	P_{n1}(a,b) & P_{n2}(a,b) &\ldots & P_{nn}(a,b)\\
	\end{pmatrix} \in SL_{n}(\mathbb{Z}), $$
	where the polynomials $P_{uv}(a,b), 1\leqslant u,v\leqslant n$ satisfy
	\begin{itemize}
		\item $P_{11}(a,b) = {lr_{i} \choose k}{ls_{i} \choose k}a^{k}b^{k} + {lr_{i} \choose k-1}{ls_{i} \choose k-1}a^{k-1}b^{k-1}+ \cdots + {lr_{i} \choose 3}{ls_{i} \choose 3}a^{3}b^{3} + {lr_{i} \choose 2}{ls_{i} \choose 2}a^{2}b^{2} + {lr_{i} \choose 1}{ls_{i} \choose 1}ab~+~1$
		
		\item $P_{12}(a,b) = {lr_{i} \choose k}{ls_{i} \choose k-1}a^{k}b^{k-1} + {lr_{i} \choose k-1}{ls_{i} \choose k-2}a^{k-1}b^{k-2}+ \cdots + {lr_{i} \choose 3}{ls_{i} \choose 2}a^{3}b^{2} + {lr_{i} \choose 2}{ls_{i} \choose 1}a^{2}b^{1} + {lr_{i} \choose 1}a $
		\\
		\ldots
		\\
		and in general
		\item $P_{uv}(a,b) = \Sigma_{u'=u,v'=v}^{u'=k+2-v,v'=k+1}{lr_{i} \choose k-u'+1}{ls_{i} \choose k-v'+1}a^{k-u'+1}b^{k-v'+1} $ 
		with $u\leqslant v$
		\item $P_{uv}(a,b) = \Sigma_{u'=u,v'=v}^{u'=k+1, v' = k+2-u}{lr_{i} \choose k-u+1}{ls_{i} \choose k-v+1}a^{k-u'+1}b^{k-v'+1}$ with $u > v$
		\\
		\ldots
		
		\item $P_{nn}(a,b) = 1$
		
	\end{itemize}
	
	Then, we have the following inequalities for $l\geqslant 3k$.
	\begin{enumerate}
		\item $1 - \frac{1}{15}\leqslant \frac{P_{11}(a,b)}{{lr_{i} \choose k}{ls_{i} \choose k}a^{k}b^{k}} \leqslant 1 + \frac{1}{15},$
		\item $ |\frac{P_{12}(a,b)}{{lr_{i} \choose k}{ls_{i} \choose k}a^{k}b^{k}}| \leqslant \frac{1}{4}(1+\frac{1}{4^{2}}+ \frac{1}{4^{4}}+\cdots+\frac{1}{4^{2k-2}}) < \frac{1}{4}.\frac{16}{15},$\\
		\\
		\ldots
		\\
		and in general
		\item $ |\frac{P_{uv}(a,b)}{{lr_{i} \choose k}{ls_{i} \choose k}a^{k}b^{k}}| \leqslant \frac{1}{4^{u+v-2}}.\frac{16}{15}\,\, \forall uv > 1.$
	\end{enumerate}
	
	We proceed as in the $n=4$ case and consider $$Z = \prod_{i=1}^{t}\mathcal{P}_{i} = \prod_{i=1}^{t} A^{lr_{i}}B^{ls_{i}}, \hbox{ 
	for some } t\in \mathbb{N}, r_{i},s_{i}\in \mathbb{Z}\backslash \lbrace 0 \rbrace .$$ \\
	We shall show, by induction, that $|z_{11}| > |z_{12}| + ... + |z_{1n}| + 1$, where the $z_{ij}, 1\leqslant i,j\leqslant n$ denote the elements of the matrix $Z$. \\
	From the above inequalities it is clear that $\Sigma_{v=2}^{n} |\frac{P_{1v}(a,b)}{{lr_{i} \choose k}{ls_{i} \choose k}a^{k}b^{k}}|<\frac{1}{2}$ which means
	$$\frac{1}{{lr_{i} \choose k}{ls_{i} \choose k}a^{k}b^{k}} + \Sigma_{v=2}^{n} \left|\frac{P_{1v}(a,b)}{{lr_{i} \choose k}{ls_{i} \choose k}a^{k}b^{k}}\right| < \frac{P_{11}(a,b)}{{lr_{i} \choose k}{ls_{i} \choose k}a^{k}b^{k}}$$
	and in turn implies that the inductive assumption (basis of induction) holds.\\
	For the main step of the induction, suppose we are already given $$Z =\begin{pmatrix}
	z_{11} & z_{12} & \ldots & z_{1n}\\
	z_{21} & z_{22} &  \ldots & z_{2n}\\
	\vdots & \vdots & &\vdots\\
	z_{n1} & z_{n2} &  \ldots & z_{nn}\\   
	\end{pmatrix}$$
	with $|z_{11}| > |z_{12}| + \cdots + |z_{1n}| + 1 $.\\
	Let $$Z' = \begin{pmatrix}
	z'_{11} & z'_{12} &  \ldots & z'_{1n}\\
	z'_{21} & z'_{22} &  \ldots & z'_{2n}\\
	\vdots&  \vdots& &  \vdots \\
	z'_{n1} & z'_{n2} &  \ldots & z'_{nn}\\   
	\end{pmatrix}  = Z\times A^{lr_{t+1}}B^{ls_{t+1}}.$$
	Expand the first row of $Z'$, i.e., $z'_{1j}, 1\leqslant j \leqslant n$ in terms of the first row of $Z$, $z_{1j}, 1\leqslant j\leqslant n$ and the elements of the matrix $ A^{lr_{t+1}}B^{ls_{t+1}}$. Then 
	we can conclude by considering the inductive assumption  $|z_{11}| > |z_{12}| + \cdots + |z_{1n}| + 1 $ and the above inequalities for the matrix $A^{lr_{t+1}}B^{ls_{t+1}}$ that 
	$$|z'_{11}| > |z'_{12}|+ \cdots+|z'_{1n}| + 1.$$
\end{proof}

%%%%%%%%%%%%%%%%%%%%%%%%%%%%%%%%%%%%%%%%%%%%%%%%%%%

\providecommand{\bysame}{\leavevmode\hbox to3em{\hrulefill}\thinspace}
\providecommand{\MR}{\relax\ifhmode\unskip\space\fi MR }
% \MRhref is called by the amsart/book/proc definition of \MR.
\providecommand{\MRhref}[2]{%
	\href{http://www.ams.org/mathscinet-getitem?mr=#1}{#2}
}
\providecommand{\href}[2]{#2}

\end{document}